\documentclass[a4paper,leqno]{amsart}

\usepackage{Kosanovic-macros}
\usepackage{amssymb,graphicx,mathtools,faktor,caption}

\usepackage{setspace}
\makeatother
  
\calclayout

\usepackage[
colorlinks=true,
citecolor={blue!80!black},
linkcolor={orange!70!purple},
urlcolor={olive!80!black},
linktoc=section
]{hyperref}

\usepackage[backend=biber,
style=alphabetic,
sorting=nyt,
giveninits=true,
maxnames=10,
backref,
backrefstyle=none,
]{biblatex}

\renewbibmacro{in:}{%
  \ifentrytype{article}{}{\printtext{\bibstring{in}\intitlepunct}}}

\DeclareFieldFormat[article,incollection,inproceedings]{title}{#1\isdot}
\DeclareFieldFormat[unpublished]{title}{\textit{#1}\isdot}

\DefineBibliographyStrings{english}{%
    backrefpage  = {p.}, 
    backrefpages = {pp.} 
}

\newtheorem{theorem}{Theorem}[section]
\newtheorem{lemma}[theorem]{Lemma}
\newtheorem{corollary}[theorem]{Corollary}
\newtheorem{alg}[theorem]{Algorithm}

\theoremstyle{definition}

\newtheorem{example}[theorem]{Example}
\newtheorem{question}[theorem]{Question}
\newtheorem{exercise}[theorem]{Exercise}

\theoremstyle{remark}
\newtheorem{remark}[theorem]{Remark}

\numberwithin{equation}{section}

\begin{document}

\title{Dancing circles}

\author{Danica Kosanovi\'c}
\address{Universit\"at Bern, Mathematisches Institut (MAI), Alpeneggstrasse 22, 3012 Bern, Switzerland}
\curraddr{CNRS, IMJ-PRG, Université Paris Cité, Sorbonne Université, F-75013 Paris, France}
\email{danica.kosanovic@cnrs.fr}
\thanks{I wish to thank the organizers of the 2025 Georgia International Topology Conference for the invitation to give a talk, Akram Alishahi and Gordana Matić for their effort in publishing these proceedings, and the referee for useful comments.}


\subjclass[2020]{Primary 57R40, 57K40, 57K45}

\date{}

\begin{abstract}
	Recent results on diffeomorphisms of $4$-manifolds make essential use of families of embeddings of $1$-manifolds into $4$-manifolds. We survey what is currently known about such families.
\end{abstract}

\maketitle

\section{Introduction}

Spaces of smooth embeddings are among the central objects in topology, the best-known example being the space $\Emb(\S^1,\R^3)$ of classical knots. 
A path between two points in this space is a smooth isotopy between knots, and knot theory studies the set $\pi_0\Emb(\S^1,\R^3)$ of path components, that is, isotopy classes of knots. Equivalently, we have a bijection $\pi_0\Emb(\S^1,\R^3)\cong\pi_0\Embp(\D^1,\D^3)$, so one can instead study the set of isotopy classes of embeddings of the arc $\D^1$ into the 3-ball, with fixed endpoints $\partial\D^1$, as on the left of Figure~\ref{fig:knot-ball}.
\begin{figure}[!htbp]
    \centering
    \includegraphics[width=0.18\linewidth]{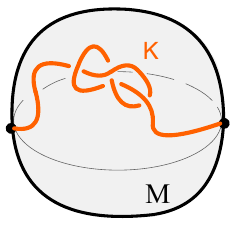}\hspace{1cm}
    \includegraphics[width=0.3\linewidth]{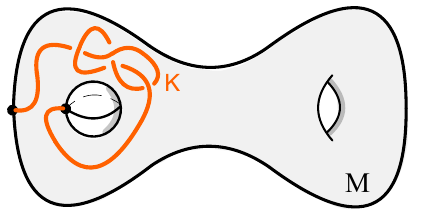}
    \caption{An embedding $K\colon C\hra M$ for $C=\D^1$ and respectively $M=\D^3$ and $M=(\S^1\tm\D^2)\setminus\D^3$.}
    \label{fig:knot-ball}
\end{figure}

More generally, one has spaces $\Embp(C,M)$ of smooth neat embeddings between compact smooth manifolds $C$ and $M$, whose points are smooth maps $K\colon C\to M$ such that $K$ is injective and for any $p\in C$ the derivative $dK|_p\colon T_pC\to T_{K(p)}M$ is injective. Moreover, we require that the boundary condition $K(C)\cap \nu\partial M=K(\nu\partial C)=e(\nu\partial C)$ is satisfied, where $\nu Y$ denotes a fixed tubular neighborhood of $Y\subset X$, and $e\colon C\hra M$ is a fixed basepoint embedding. This space is equipped with the Whitney $C^\infty$-topology: nearby embeddings are those whose derivatives are uniformly close.

Instead of just the set of path components of $\Embp(C,M)$ we can study its fundamental group and higher homotopy groups $\pi_n(\Embp(C,M), e)$, $n\geq1$, or (co)homology. In other words, in this general setting a nontrivial ``knot'' is a nontrivial homotopy class represented by an $n$-parameter family of embeddings of $C$ into $M$ which cannot be trivialized through such families. Thus, all embeddings are individually isotopic to $e$, but the whole family cannot be homotoped to the family $\const_e$, that is constantly equal to $e$.
Let us illustrate this with examples. 

\begin{example}[The trefoil and its generalizations]\label{ex:intro}
    \hfill
\begin{enumerate}
    \item (Classical.)
The (right-handed) trefoil knot
\[
    T\colon\D^1\hra\D^3
\]
can be obtained from the standard inclusion $\unkn\colon\D^1\hra\D^3$, called the unknot, by picking three subarcs of $\unkn$ and connect-summing them into the Borromean rings, as on the right of Figure~\ref{fig:Borr}. Equivalently, fix the meridians $m_1,m_2\colon\S^1\hra\D^3\setminus\unkn(\D^1)$ to $\unkn$ at two different points, and define $T$ by replacing a subarc of $\unkn$ with the arc describing the commutator $[m_1,m_2]\coloneqq m_1m_2m_1^{-1}m_2^{-1}$, as on the left of Figure~\ref{fig:Borr}.
\begin{figure}[!htbp]
    \centering
    \includegraphics[width=0.65\linewidth]{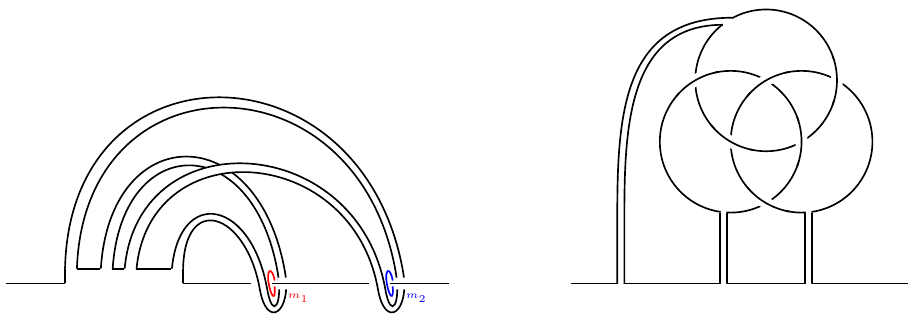}
    \caption{Two isotopic knots, obtained by degree $2$ clasper surgery on the unknot.}
    \label{fig:Borr}
\end{figure}
\item (Higher dimensional.)
    Fix $k\geq1$. Start with the standard embedding $e\colon\D^{4k-1}\hra\D^{6k}$, and pick two meridians $m_1,m_2\colon\S^{2k}\hra\D^{6k}\setminus e(\D^{4k-1})$, defined as the boundaries of normal disks at two points of $e(\D^{4k-1})$. Consider their Whitehead product 
    \[\begin{tikzcd}
        {[m_1,m_2]}\colon\S^{4k-1}\rar{Wh}
        & \S^{2k}\vee\S^{2k}\rar{m_1\vee m_2}
        & \D^{6k}\setminus e(\D^{4k-1}),
    \end{tikzcd}
    \]
    where the universal Whitehead product $Wh$ is defined as the attaching map of the top $4k$-cell in $\S^{2k}\times\S^{2k}$. It turns out that $[m_1,m_2]$ is homotopic to an embedding 
    \[
        [m_1,m_2]_{emb}\colon\S^{4k-1}\hra\D^{6k}\setminus e(\D^{4k-1}),
    \]
    as shown by Haefliger~\cite{Haefliger-knotted}.
    We can thus tube $e$ (ambiently connect sum) into the sphere $[m_1,m_2]_{emb}$ to obtain a new embedding $T\colon\D^{4k-1}\hra\D^{6k}$. This is precisely the Haefliger trefoil~\cite{Haefliger-codim-high}, which generates the group of isotopy classes 
    \[
        \pi_0\Embp(\D^{4k-1},\D^{6k})\cong\Z.
    \]
    In contrast, since the codimension is $2k+1>2$, this group is trivial in the $\mathsf{PL}$ and $\mathsf{Top}$ categories by Zeeman~\cite{Zeeman}.

\item (Dancing circles.)
    Fix $d\geq4$. A meridian of $\unkn\colon\D^1\hra \D^d$ is a sphere of dimension $d-2$. Thus, the Whitehead product of two meridians is a map
    \[\begin{tikzcd}
        {[m_1,m_2]}\colon\S^{2(d-3)+1}\rar{Wh}
        & \S^{d-2}\vee\S^{d-2}\rar{m_1\vee m_2}
        & \D^d\setminus \unkn(\D^1).
    \end{tikzcd}
    \]
    We can foliate this sphere by embedded circles and ambiently connect-sum a fixed subarc of $\unkn$ into each of these circles to obtain a $2(d-3)$-parameter family of embeddings of $\D^1$. We made this idea precise in~\cite{K-families} by introducing \emph{graspers}. It turns out that this trefoil grasper family generates
    \[
        \pi_{2(d-3)}(\Embp(\D^1,\D^d),\unkn)\cong\Z.\qedhere
    \]
\end{enumerate}
\end{example}

In this survey we consider the setting (3) for the case $d=4$. More generally, we will discuss for all $n\geq1$ the groups
\[
        \pi_n\big(\Embp(C,M),e\big)
\]
where $C=\D^1$ or $C=\S^1$ and $M$ is a $4$-dimensional smooth manifold.
This somewhat underdeveloped version of knot theory has gained more attention recently, as we briefly indicate below. We will see that this topic unites the viewpoints of low-dimensional topology on the one hand, and algebraic topology and homotopy theory, on the other. As sketched in §\ref{sec:beyond}, the algebraic side can benefit from geometric constructions, and the geometric side from the powerful formal language.

\subsection*{One motivation: diffeomorphisms}
Let us mention one recent result that motivates the study of spaces of embeddings of arcs and circles.

Not a single mapping class group of a smooth 4-manifold has been computed. Notably, the question of nontriviality of $\pi_0\Diff(\S^4)$ remains open. However, we do know that $\Diff(\S^4)$ is not homotopy equivalent to its subgroup of orthogonal maps $O(5)$, unlike its lower dimensional analogues $\Diff(\S^d)\simeq O(d+1)$ for $d\leq 3$. This is due to Watanabe~\cite{Watanabe-dim-4}, who in particular showed that $\pi_1(\Diff(\S^4);\Id)$ has an infinite order element. In fact, Watanabe's construction also works in degree zero, giving a candidate class
\[
    \Theta\in\pi_0\Diff(\S^4).
\]
It has been shown by Gay and Hartman~\cite{Gay-Hartman} that this class is either trivial or 2-torsion $\Theta^2=[\Id]$, and independently by the author~\cite{K-diffeos-from-graspers} and Gay~\cite{Gay-theta} that it agrees with the more recent construction due to Budney and Gabai~\cite{Budney-Gabai} in terms of barbells. 
In~\cite{K-diffeos-from-graspers} such barbell diffeomorphisms of a $4$-manifold $\X$ are described as those in the image of the \emph{parametrized surgery map} 
\[
    \ps\colon
\pi_1\big(\Emb(\nu\S^1,X\#\S^1\tm\S^3);\nu\cknot\big)
	\ra
\pi_0\Diffp(X\#\S^1\tm\D^3)
    \ra
\pi_0\Diffp(X).
\]
The source is the fundamental group of the space of framed embeddings of $\S^1$ into $X\#\S^1\tm\S^3$, based at $\nu c=\S^1\tm\D^3\subset\S^1\tm\S^3$. The first map is given by ambient isotopy extension followed by restriction to the complement of $\nu\cknot$. The second map extends by the identity over $\D^2\tm\S^2$, glued to the $\S^1\tm\D^3$ factor so that $X\#\S^4\cong X$.

In this breakthrough work Budney and Gabai consider the homomorphisms
    \begin{equation}\label{eq:BG}
    \begin{tikzcd}[scale cd=0.95, column sep=small]
        \pi_1\big(\Emb(\nu\S^1,\S^1\tm\D^3\#\S^1\tm\S^3);\nu\cknot\big)
            \rar{\ps}
        & \pi_0\Diffp(\S^1\tm\D^3)
            \dar{\res_D} 
        &\\
        & \pi_0\Embp(\D^3,\S^1\tm\D^3)
            \rar{\fol_2}
        & \pi_2\Embp(\D^1,\S^1\tm\D^3).
    \end{tikzcd}
    \end{equation}
The map $\res_D$ is the restriction of the diffeomorphism to $\{pt\}\tm\D^3\subset \S^1\tm\D^3$, whereas $\fol_2$ foliates the given embedded 3-ball into a 2-parameter family of embedded arcs.
 
\begin{theorem}[Budney--Gabai~\cite{Budney-Gabai}]
    The sequence of maps~\eqref{eq:BG} has an infinitely generated free abelian subgroup in its image. In particular, we have $\Z^\infty<\pi_0\Diffp(\S^1\tm\D^3)$ and $\Z^\infty<\pi_0\Embp(\D^3,\S^1\tm\D^3)$.
\end{theorem}

\subsection*{Literature}
We will survey recent developments in the study of spaces of smooth embeddings of arcs and circles, based on the classical work of Smale~\cite{Smale}, Whitney~\cite{Whitney-n-in-2n}, Wall~\cite{Wall-surgery}, Haefliger~\cite{Haefliger-plongements}, and Dax~\cite{Dax}. 

We first focus on computations of the fundamental group, obtained recently by Gabai~\cite{Gabai-disks}, the author~\cite{K-Dax,K-Dax2}, and the author with Teichner~\cite{KT-highd,KT-4dLBT}. Related work was done by Schwartz~\cite{Schwartz} and Lin--Wu--Xie--Zhang~\cite{Lin-Wu-Xie-Zhang}, and in the topological category by Cha--Kim~\cite{Cha-Kim}. In particular, we explain how to compute $\pi_1(\Emb(\S^1,X\#\S^1\tm\S^3);\cknot)$ --- see Corollary~\ref{cor:main}. For instance, we will compute $\pi_1(\Emb(\S^1,\S^1\tm\D^3\#\S^1\tm\S^3);\cknot)$ in Example~\ref{ex:intro-circles}. For computations of framed 1-families we refer to \cite{K-Dax2}.

In the last section we discuss an approach to higher homotopy groups of these embedding spaces due to Goodwillie, Klein, and Weiss~\cite{GW,GKW,GKmultiple}, and its developments by the author in~\cite{K-thesis-paper,K-families}, based on the Gusarov--Habiro theory~\cite{Gusarov-main,Gusarov,Habiro} and the work of Conant and Teichner~\cite{CT1}.

\subsection*{Outline}

In §\ref{sec:imm} we summarize what is known about immersion spaces $\Immp(\D^1,\X)$ and $\Imm(\S^1,N)$ for smooth $d$-manifolds $\X$ and $N$, with $\partial X\neq\emptyset$. In particular, in §\ref{subsec:imm-arcs} we compute
\[
\pi_n(\Immp(\D^1,\X),\unkn)\cong\pi_{n+1}\S^3\oplus\pi_{n+1}\X,
\]
and in §\ref{subsec:imm-circles} we express $\pi_1(\Imm(\S^1,N),\cknot)$ as an extension of the subgroup $\Fix_{[\cknot]}(\pi_1N)$ of the elements of $\pi_1N$ fixed by $[\cknot]\in\pi_1N$, by the abelian group ${\pi_2N}/{b\sim[\cknot]\cdot b}$.

After discussing Wall's self-intersection invariants $\mu$ in §\ref{subsec:wall} for inspiration, in §\ref{subsec:Dax} we define the {Dax invariant} $\Dax$ of regular nullhomotopies. 

In §\ref{sec:fund-gp} we use $\Dax$ to compute the fundamental groups of embedding spaces $\Embp(\D^1,\X)$ and $\Emb(\S^1,N)$. For this we also need the Dax homomorphism $\dax$ from §\ref{subsec:dax-homomorphism} and its variant $\dax^{whisk}$ from §\ref{subsec:dax-whisk}. In §\ref{subsec:Dax-Wall} we point out an interesting relation between $\dax$ and the Wall invariants, discovered in~\cite{KT-4dLBT}.

In §\ref{sec:examples} we present example computations of the fundamental group of embedding spaces, both for arcs and circles.

In §\ref{sec:beyond} we discuss approaches to higher homotopy groups of these embedding spaces: in §\ref{subsec:graspers} we define graspers, in §\ref{subsec:realmap} {grasper} families, and in §\ref{subsec:emb-calc} we briefly introduce {embedding calculus} of Goodwillie and Weiss, and explain how the preceding results fit into that framework.

Finally, in Appendix~\ref{app:Smale} we summarize the proof of Smale's equivalence for immersions of arcs, and in Appendix~\ref{app:WHD} we summarize the history of results in the ``metastable range''.

\subsection*{Figures used}

The following figures are borrowed from the author's previous papers:
Figures~\ref{fig:Borr} and~\ref{fig:grasper} from~\cite{K-families}, Figure~\ref{fig:realmap-Dax} from~\cite{K-diffeos-from-graspers}, Figures~\ref{fig:u-whiskered} and~\ref{fig:Phi} from~\cite{K-Dax}. However, Figures~\ref{fig:knot-ball},~\ref{fig:g+g-inv},~\ref{fig:compute}, and~\ref{fig:tree} appear here for the first time.

\subsection*{Tools used} 

AI tools were used at the final stage of preparing this manuscript to identify grammatical mistakes, notational inconsistencies, and minor mathematical errors.

\section{Immersions}\label{sec:imm}

Consider the space $\Immp(C,M)$ of smooth immersions, whose points are smooth maps $K\colon C\to M$ such that for any $p\in C$ the derivative $dK|_p\colon T_pC\to T_{K(p)}M$ is injective, and we have $K(C)\cap \nu(\partial M)=K(\nu(\partial C))=e(\nu(\partial C))$. The only difference from the space of embeddings lies in dropping the injectivity of the map $K$. In particular, we have inclusions
\begin{equation}\label{eq:incl}
    \Embp(C,M)\hra \Immp(C,M)\hra C^\infty(C,M).
\end{equation}
For a smooth map $K$ to be an immersion we only need to check the \textit{local} condition on injectivity of its derivative at each point. On the other hand, for an immersion $K$ to be an embedding we need to check the \textit{global} condition on the injectivity of $K$: are there two points $p_1,p_2\in C$ with $p_1\neq p_2$ but $K(p_1)=K(p_2)$?

It turns out that this key difference makes immersions amenable to the techniques of algebraic topology, whereas embeddings seem to evade them --- see, however, §\ref{subsec:emb-calc}.

\subsection{Immersed arcs}\label{subsec:imm-arcs}

The first breakthrough in the study of immersions was Smale's paper~\cite{Smale}, which showed the existence of a homotopy equivalence
\begin{equation}\label{eq:Smale}
    \begin{tikzcd}
    \Immp(\D^1,X)\rar{\deriv}[below]{\sim}
    & \Map_\partial(\D^1,\S TX).
\end{tikzcd}
\end{equation}
On the left-hand side we have the space of immersions of the interval into a smooth $d$-manifold $\X$, which near the boundary agree with the basepoint immersion $\unkn\colon\D^1\imra X$. For $\unkn$ we assume $\unkn(\pm1)=x_{\pm}\in\partial X$. On the right-hand side we have the space of paths in the unit tangent bundle of $\X$ that begin and end respectively at the unit derivative of $\unkn$ at $-1$ and $1$. The map from the left to the right is the unit derivative (the derivative vector scaled to have norm $1$). We outline the proof of \eqref{eq:Smale} in Appendix~\ref{app:Smale}.

Moreover, there is a homotopy equivalence to the based loop space
\begin{equation}\label{eq:loop-space}
    \begin{tikzcd}
    -\cdot(\deriv\unkn)^{-1}\colon\Map_\partial(\D^1,\S TX)\rar
    & \Omega (\S TX)
\end{tikzcd}
\end{equation}
by concatenating a path in $\S TX$ with the reverse of the path given by the derivatives of $\unkn$; the map $\gamma\mapsto\gamma\cdot\deriv\unkn$ defines a homotopy inverse.

The space $\S TX$ is a bundle over $\X$ with fiber $\S^{d-1}$. This bundle admits a section because the Euler class vanishes, as $\X$ has nonempty boundary. This implies that there is a homotopy equivalence
\[
    \Omega (\S TX)\simeq \Omega (X \tm \S^{d-1}).
\]

\begin{corollary}\label{cor:imarcs}
    Sending $k\in\ImArcs{X}$ to the loop $p_{\X,\unkn}(k)\coloneqq k\cdot\unkn^{-1}$ based at $\unkn(-1)$ gives a map
\begin{equation}\label{eq:pu}
    p_{\X,\unkn}\colon\ImArcs{X}\to\Omega X,
\end{equation}
which in degrees $n\leq d-3$ induces isomorphisms 
\[
	\pi_n(\Immp(\D^1,\X),\unkn)\cong\pi_n(\Omega X,\const_{\unkn(-1)})\cong\pi_{n+1}(X,\unkn(-1)),
\] 
whereas in degree $n=d-2$ we have
\[
    \pi_{d-2}(\Immp(\D^1,\X),\unkn)\cong\Z\oplus\pi_{d-1}\X.
\]
\end{corollary}

Of particular relevance for us is the following case.
\begin{corollary}
    For a smooth 4-manifold $\X$ with nonempty boundary we have
    \[
        \pi_0\ImArcs{\X}\cong\pi_1\X,\quad
        \pi_1\ImArcs{\X}\cong\pi_2\X,\quad
        \pi_2\ImArcs{\X}\cong\Z\oplus\pi_3\X.
    \]
\end{corollary}

\subsection{Immersed circles}\label{subsec:imm-circles}

Now fix a smooth $d$-manifold $N$, possibly with empty boundary.
It is not hard to derive the homotopy equivalence
\begin{equation}\label{eq:Smale-circles}
    \begin{tikzcd}
    \Imm(\S^1,N)\rar{\deriv}[below]{\sim}
    & \Map(\S^1,\S TN)
\end{tikzcd}
\end{equation}
from Smale's equivalence~\eqref{eq:Smale}; see \cite[Prop.4.1]{K-Dax}. Alternatively, this follows from the general Smale--Hirsch h-principle for immersions.
\begin{remark}
    It is worth mentioning that Whitney~\cite{Whitney-curves-in-plane} classified immersed circles in the plane up to regular homotopy, the result known as Whitney--Graustein Theorem. We will see how it follows from~\eqref{eq:Smale-circles} in Exercise~\ref{ex:Whitney-Graustein1}.
\end{remark}

Homotopy groups of free loop spaces like $\Map(\S^1,\S TN)$ can be expressed as follows. Let $[\deriv\cknot]\in\pi_1\S TN$ denote the homotopy class of $\deriv\cknot$, and $[\cknot]\in\pi_1N$ the homotopy class of the loop $\cknot$. Let $[\deriv\cknot]\cdot b$ denote the usual action of $[\deriv\cknot]\in\pi_1\S TN$ on $b\in\pi_n\S TN$; for $n=1$ this is just the conjugation action $[\deriv\cknot]\cdot b=[\deriv\cknot]\, b\, [\deriv\cknot]^{-1}$. Finally, denote 
\[
    \Fix_{[\deriv\cknot]}(\pi_n\S TN)\coloneqq \big\langle b\in\pi_n\S TN: b=[\deriv\cknot]\cdot b\big\rangle.
\]
\begin{lemma}[{\cite[Cor.4.2]{K-Dax}}]\label{lem:imm-htpy-gps}
   There is a bijection between $\pi_0\Imm(\S^1,N)$ and the set of orbits of the conjugation action on $\pi_1\S TN$, and for all $n\geq1$ there are group extensions
    \[\begin{tikzcd}
        \faktor{\pi_{n+1}\S TN}{b\sim[\deriv\cknot]\cdot b}\rar[tail] 
        &  \pi_n(\Imm(\S^1,N),\cknot)\rar[two heads] 
        & \Fix_{[\deriv\cknot]}(\pi_n\S TN).
    \end{tikzcd}\]
    Moreover, for $n\leq d-2$ we have $\pi_n\S TN\cong\pi_nN$.
\end{lemma}

In particular, the subgroup $\Fix_{[\cknot]}(\pi_1N)\coloneqq \big\langle b\in\pi_1N: b=[\cknot] b [\cknot]^{-1} \big\rangle < \pi_1N$ is the centralizer of $[\cknot]\in\pi_1N$.  
Note that if $b\in \Fix_{[\cknot]}(\pi_1N)$ and $b=[\beta]$ for some loop $\beta$, then the map $\beta\vee \cknot\colon\S^1\vee\S^1\to N$ extends to a map $\beta\tm \cknot\colon \S^1\tm\S^1\to N$. Foliating this torus yields a 1-parameter family of immersed circles based at $\cknot$.

\begin{corollary}[{\cite[Thm.2.7]{K-Dax2}}]\label{cor:imcircles}
    For a connected smooth 4-manifold $N$, there is a bijection between $\pi_0\Imm(\S^1,N)$ and the set of orbits of the conjugation action on $\pi_1\S TN$, and for any $\cknot\in\Imm(\S^1,N)$ there are group extensions
	\[
    \begin{tikzcd}[row sep=tiny]
        \faktor{\pi_2N}{b\sim[\cknot]\cdot b}\rar[tail] 
        & \pi_1(\Imm(\S^1,N),\cknot)\rar[two heads] 
        & \Fix_{[\cknot]}(\pi_1N),\\
        \faktor{\pi_3\S TN}{b\sim [\deriv\cknot]\cdot b} \rar[tail] 
        & \pi_2(\Imm(\S^1,N),\cknot) \rar[two heads] 
        & \Fix_{[\cknot]}(\pi_2N).
    \end{tikzcd}
	\]
\end{corollary}

\begin{remark}\label{rem:atlas}
	We will use the ATLAS notation $G\cong N.H$ to denote an \textit{arbitrary} group extension $N\hra G\sra H$ (possibly split $G\cong N\rtimes H$ or trivial $G\cong N\times H$). In particular, from the corollary we have
\begin{equation}\label{eq:pi1-imm}
    \pi_1(\Imm(\S^1,N),\cknot)\cong\Big(\faktor{\pi_2N}{b\sim[\cknot]\cdot b}\Big).\Big(\Fix_{[\cknot]}(\pi_1N)\Big).
\end{equation}
\end{remark}

\begin{remark}\label{rem:relations}
	We use the notation $x\sim y$ for families of relations in which $\X$ and $y$ range over elements of the ambient group subject to conditions specified by the context. For example, in~\eqref{eq:pi1-imm}, $b\in\pi_2N$ is an arbitrary element, whereas $[\cknot]\cdot b$ is the action of $[\cknot]\in\pi_1N$ on it.
\end{remark}

\subsection{Examples}

\begin{exercise}\label{ex:Whitney-Graustein1}
    From Smale's equivalence~\eqref{eq:Smale-circles}, Lemma~\ref{lem:imm-htpy-gps}, and the equivalence $\S T\R^2\simeq \S^1$, we can derive the Whitney--Graustein Theorem:
    \[
        \pi_0\Imm(\S^1,\R^2) \cong \faktor{\pi_1(\S T\R^2)}{\sim}
        \cong \faktor{\pi_1\S^1}{\sim} \cong\Z.
    \]
    Identify this isomorphism with Whitney's total rotation number~\cite{Whitney-curves-in-plane}.
\end{exercise}

\begin{exercise}\label{ex:Whitney-Graustein2}    
    More generally, for any surface $N\neq\S^2,\R P^2$ show that
    \[
        \pi_0\Imm(\S^1,N) \cong \faktor{\pi_1(\S TN)}{\sim}
        \cong \Z.\pi_1N.
    \]
    On the other hand, $\pi_0\Imm(\S^1,\S^2) \cong \Z/2$ and $\pi_0\Imm(\S^1,\R P^2) \cong \Z/2.\Z/2$.

    \emph{Hint.} Use the long exact sequence for the bundle $\S^1\hra \S TN\to N$, whose connecting map $e|_{\pi_2N}\colon\pi_2N\to\pi_1\S^1\cong\Z$ is given by the restriction of the Euler class  (equivalently, the Euler class of the universal cover $\wt{N}$).
    Furthermore, in the resulting extension $\Z/e\hra \pi_1(\S TN)\sra\pi_1N$ the action of $\pi_1N$ on $\Z/e$ is nontrivial exactly on orientation-reversing curves. This means that if two immersed circles are homotopic, there is an obstruction in $\Z/e$ for them to be regularly homotopic.    
    This can be identified as the Reinhart--Chillingworth rotation number~\cite{Chillingworth}, or as the count of cusps in a homotopy between the curves as shown by Quine~\cite{Quine}.
\end{exercise}

\begin{exercise}\label{ex:self-centralizing}
    For any $4$-manifold $N$, show that if 
    $\Fix_{[\cknot]}(\pi_1N)=\langle[\cknot]\rangle$, then the extension~\eqref{eq:pi1-imm} is trivial:
\[
   \pi_1(\Imm(\S^1,N),\cknot)\cong\pi_1(\Map(\S^1, N),\cknot)\cong \langle [\cknot]\rangle\tm\faktor{\pi_2N}{b\sim[\cknot]\cdot b}.
   \qedhere
\]
\emph{Hint.} \cite[Thm.2.7]{K-Dax2}.
\end{exercise}

\begin{exercise}\label{ex:nullhtpic}
    For any $4$-manifold $N$, show that if $\cknot$ is nullhomotopic, then the extension~\eqref{eq:pi1-imm} splits. More precisely, we have the semi-direct product
\[
   \pi_1(\Imm(\S^1,N),\cknot)\cong \pi_2N\rtimes\pi_1N,
\]
    with the usual action of $\pi_1N$ on $\pi_2N$. 
    
    \emph{Hint.} In this case the map $\Imm(\S^1,N)\to \S TN$, which evaluates a loop at a basepointm has a section given by small loops.
\end{exercise}

\section{From immersions to embeddings}\label{sec:imm-emb}

In §\ref{subsec:wall} we review self-intersection invariants, which obstruct a single immersion from being (regularly) homotopic to an embedding. In §\ref{subsec:Dax} we explain how these invariants generalize to the Dax invariant, which detects loops of embeddings that are regularly nullhomotopic.

\subsection{Self-intersection invariants}\label{subsec:wall}

Assume now that $\dim M=2$. Then a generic map $f\colon \D^1\to M$, with boundary fixed in $\partial M$, is an immersion with a discrete set of transverse double points. 

\begin{question}
    When is $f$ homotopic to an embedding?
\end{question}

This was studied by Whitney~\cite{Whitney-n-in-2n} in the context of maps of manifolds to Euclidean spaces (see Appendix~\ref{app:Whitney}), and by Wall~\cite{Wall-surgery} who defined the self-intersection invariant $\mu(f)\in \Z[\pi_1M]$
as the sum of {signed double point loops}:
\[
    \mu(f)\coloneqq\sum_{\text{d.p. }p} \e_p[\gamma_p]\in \Z[\pi_1M].
\]
That is, for every double point $p=f(x_1)=f(x_2)$ we look at the loop given by
\begin{equation}\label{eq:dpl}
    \gamma_p\coloneqq f([-1,x_1])\cdot f([-1,x_2])^{-1}.
\end{equation}
In other words, we start at the basepoint $f(-1)\in\partial M$ and follow $f$ until we encounter $p$ for the first time, say at $p=f(x_1)$; then we switch to the other sheet $p=f(x_2)$ and return to $f(-1)$ following $f$.

Additionally, the double point $p$ carries a sign $\e_p$ coming from comparing the orientation of $T_pM$ and
\begin{equation}\label{eq:dps}
    df(T_{x_1}\D^1)\oplus df(T_{x_2}\D^1).
\end{equation}
One can show that $\mu(f)$ stays invariant under regular homotopies, that is, homotopies through immersions. If we allow all homotopies, however, then $\mu(f)$ can change only by $\pm1$, because of local cusp homotopies. Therefore, let us remove the coefficient of the trivial group element $1\in\pi_1M$ and denote the result by
\begin{equation}
    \ol{\mu}(f)\in \Z[\pi_1M\setminus1].
\end{equation}

These invariants can also be defined for any immersion $f\colon C\imra M$ such that
\[
    \dim M=2\dim C,
\]
with some modifications. We need to assume that $C$ is simply connected and we need to quotient the target $\Z[\pi_1M]$ by 
\[
    g-(-1)^{\dim C}g^{-1}
\]
for all $g\in\pi_1M$. Namely, in general we cannot order the two preimages as $x_1<x_2$, as in the case of $C=\D^1$: this ambiguous sheet choice reverses the loop $\gamma_p$ in~\eqref{eq:dpl}, and swaps the two summands in~\eqref{eq:dps} thus changing $\e_p$ to $(-1)^{\dim C}\e_p$.

\begin{theorem}[Wall~\cite{Wall-surgery}, Haefliger~\cite{Haefliger-plongements}] 
Assume $\dim M=2\dim C$.
\begin{enumerate}
\item 
    If an immersion $f\colon C\imra M$ is regularly homotopic to an embedding, then
    \[
        \mu(f)=0\quad \in\quad \faktor{\Z[\pi_1M]}{g\sim (-1)^{\dim C}g^{-1}}.
    \]
\item 
    If an immersion $f\colon C\imra M$ is homotopic to an embedding then 
    \[
        \ol{\mu}(f)=0\quad \in\quad \faktor{\Z[\pi_1M]}{1\sim0,\, g\sim (-1)^{\dim C}g^{-1}}.
    \]
\item
    Both converse statements hold if $\dim C\geq3$.
\end{enumerate}
\end{theorem}

The converse statements are proven using the Whitney trick. In the cases $(\dim C,\dim M)=(1,2)$ and $(2,4)$ we cannot find embedded Whitney disks in general. For interesting work on these cases see for example \cite{Birman-Series} and \cite{Conant-Schneiderman-Teichner}.

Haefliger's work slightly predates Wall's, and is in fact more general: his main result, which we state as Theorem~\ref{thm:Haef} in Appendix~\ref{app:Haefliger}, classifies all embeddings homotopic to the given $f$.

\subsection{The Dax invariant}\label{subsec:Dax}

The work of Haefliger and Wall was extended independently by Dax \cite{Dax} and by Hatcher and Quinn~\cite{Hatcher-Quinn}. They extend the previous $\pi_0$ statement to a computation of homotopy groups of embedding spaces in a range, under the dimension assumption:
\[
    \dim C< \frac{2}{3}\dim M.
\]
For the conceptual explanation of this bound and for general dimensions we refer to Appendix~\ref{app:Dax}, whereas here we concentrate on
\[
    (\dim C,\dim M)=(1,4).
\]
Let us fix smooth $4$-manifolds $\X$ and $N$, together with basepoint embeddings $\unkn\colon\D^1\hra\X$ and $\cknot\colon\S^1\hra N$. We claim that the inclusions of embeddings into immersions induce surjections
    \begin{equation}\label{eq:emb-into-imm}
        \begin{tikzcd}[row sep=tiny]
        \pi_1(\Emb(\D^1,\X), \unkn)\rar[two heads]{\iota_{\X}} 
        & \pi_1(\Imm(\D^1,\X), \unkn),\\
        \pi_1(\Emb(\S^1,N), \cknot)\rar[two heads]{\iota_N} 
        & \pi_1(\Imm(\S^1,N), \cknot).
        \end{tikzcd}
    \end{equation}
Indeed, given an immersion $f\colon \D^1\imra X$, at each double point we can find a normal direction in which we can resolve it. Similarly, for a loop of immersions $f_t\colon \D^1\imra X$, $f_0=f_1$, the trace
\[
    \wt{f}\colon \D^1\tm[0,1]\ra X\tm[0,1]
\]
is an immersion of a 2-manifold into a 5-manifold. At each self-intersection point of $\wt{f}$ there is one direction normal to both sheets; this is computed as $2(5-2)-5=1$. After an isotopy, we can make this vector normal to the time axis. Thus, pushing off one sheet in this direction removes the self-intersection while keeping it the trace of a 1-parameter family of immersions $\D^1\imra X$. Therefore, $(f_t)_{t\in[0,1]}$ can be perturbed to a family $(\wh{f}_t)_{t\in[0,1]}$ of embeddings. The same holds for $\S^1\imra N$.

Using Corollaries~\ref{cor:imarcs} and~\ref{cor:imcircles} the maps~\eqref{eq:emb-into-imm} transform into
    \begin{equation}\label{eq:pi1p}
        \begin{tikzcd}[row sep=tiny]
        \pi_1(\Emb(\D^1,\X), \unkn)\rar[two heads]{\pi_1 p_{\X,\unkn}} 
        & \pi_2\X,\\
        \pi_1(\Emb(\S^1,N), \cknot)\rar[two heads]{\pi_1 p_{N,\cknot}} 
        & \Big(\faktor{\pi_2N}{b\sim[\cknot]\cdot b}\Big).\Big(\Fix_{[\cknot]}(\pi_1N)\Big).
        \end{tikzcd}
    \end{equation}

The question now becomes the following.
\begin{question}
	How many choices of lifts $(\wh{f}_t)_{t\in[0,1]}$ are there? That is, what are the kernels of $\pi_1p_{\X,\unkn}$ and $\pi_1 p_{N,\cknot}$?
\end{question}
This is answered by the following special case of Dax's work, first studied by Gabai~\cite{Gabai-disks} in the context of the failure of light bulb tricks for disks with duals.

Observe that for a class $f$ in the kernel of these maps, represented by a loop of embeddings, we can find a homotopy through immersions to the constant loop based at $e=\unkn$ or $e=\cknot$ respectively, that is, a regular nullhomotopy of $f$. That is, we have a map
\begin{equation}\label{eq:homotopy}
    F\colon[0,1]\to\Omega\Immp(C,M),\quad \tau\mapsto F_{-,\tau},\quad F_{-,0}=f,\quad F_{-,1}=\const_e,
\end{equation}
where either
\[
    (e,C,M)=(\unkn,\D^1,X)
    \quad \text{or}\quad 
    (e,C,M)=(\cknot,\S^1,N).
\]
Equivalently, we can describe $F\colon [0,1]^2\to \Imm(C,M)$ by a square 
\[
\begin{tikzcd}
      e\rar[equals]\dar[equals]\ar[phantom]{dr}{F} 
      & e\dar[equals] \\
      e\ar[r,"f" description]
      & e
\end{tikzcd}
\]
The Dax invariant $\Dax(F)$ counts the signed double points of immersions in $F$, using the following Algorithm. 

\begin{alg}\label{alg}\hfill
\begin{enumerate}
    \item Pick $F$ as in \eqref{eq:homotopy} such that $F_{s,\tau}\colon C\imra M$ fails to be embedded only at finitely many times $\tau=\tau_j\in[0,1]$ and $s=s_j\in[0,1]$. Note that $F_{0,\tau}=F_{1,\tau}=e$ for all $\tau\in[0,1]$.
    \item For every $j$ make an ordered list
    \[
    *<x^j_1<\dots<x^j_{k_j}
    \]
    of double points $x^j_i\in M$ of the nonembedded 1-manifold $F_{s_j,\tau_j}(C)$, using the positive orientation of $C$ and the basepoint $-1\in C$ for the ordering. Note that at each $x_i^j$ there are two sheets (local pieces of $F_{s_j,\tau_j}$ passing through $x_i^j$), and we can order them as well.
    \item Modify $F$ so that for each $x_i^j$:
    \begin{enumerate}
        \item only one of the sheets moves with $\tau$ or $s$, whereas the other stands still in both of these directions, 
        \item there is a chart $\R\times\R^3\subset M$ around $x_i^j$, so that the $s$-derivative of the moving sheet is parallel to the positive $\R\times\{\Vec{0}\}$-direction (that is, moves from past into future with $s$).
    \end{enumerate}
    \item 
For each $x_i^j$ compute the double point \textbf{loop}
    \[
        h_i^j= [\gamma_1\cdot\gamma_2]\in\pi_1M
    \]
    as follows: $\gamma_1$ starts at $e(-1)=F_{s_j,\tau_j}(-1)$ and goes on $F_{s_j,\tau_j}$ until it reaches $x_i^j$ for the \emph{first time}, and $\gamma_2$ starts at $x_i^j$ and follows the \emph{other sheet} with the \emph{reversed orientation} until it reaches $e(-1)$.
    \item 
For each $x_i^j$ compute the double point \textbf{sign} 
    \[
    \epsilon(x_i^j)=\epsilon_1\cdot\epsilon_2\in\{\pm1\}
    \]
    as follows: 
    $\epsilon_1=1$ (resp.\ $\epsilon_1=-1$) if the first (resp.\ second) sheet is moving, and 
    $\epsilon_2$ is positive if and only if the following basis of $\{0\}\times\R^3$ is positively oriented: the first vector is the derivative in the $\tau$-direction of the moving sheet, the second vector is tangent to the moving sheet, the third vector is tangent to the non-moving sheet.
    \item Finally, define
    \[
    \Dax(F)\coloneqq\sum_j\sum_i\epsilon(x_i^j)h_i^j.\qedhere
    \]
\end{enumerate}
\end{alg}
\begin{figure}[!htbp]
    \centering
    \includegraphics[width=\linewidth]{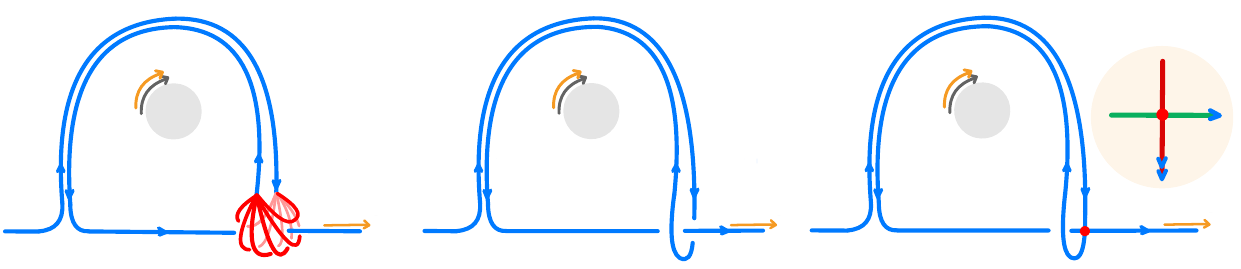}
    \caption{
        (i)~The family $\prealmap(g)$. 
        (ii)~The member of this family completely contained in the present $\{0\}\times\R^3$. 
        (iii)~The only immersed member of the family $\prealmap(g)_{s,\tau}$, and the local picture at its double point.
    }
    \label{fig:realmap-Dax}
\end{figure}

The following exercise will be of key importance, see Theorem~\ref{thm:Dax-dax}.

\begin{exercise}\label{exerc:realmap}
    For $g\in\pi_1M$ we define a family $\prealmap(g)=(\prealmap(g)_t\colon C\hra M)_{t\in[0,1]}$ of arcs or circles, based at $\prealmap(g)_0=\prealmap(g)_1=e$. It is depicted in Figure~\ref{fig:realmap-Dax}(i) and described as follows.
\begin{enumerate}
        \item 
    During the time $t\in[0,1/4]$ one piece of $e$ is dragged around a loop representing $g$, until it reaches a neighborhood of another point $p\in e(C)$.
        \item 
    Then for $t\in[1/4,3/4]$ the tip of the arc gets twirled around the meridian sphere to $e$ at $p$.
        \item 
    For $t\in[3/4,1]$ the arc gets dragged back to its initial position.
\end{enumerate}
    If we think of the meridian sphere to $e$ at $p$ as the union of a circle in the present slice together with one disk in the past and one in the future, then the embedding $\prealmap(g)_{1/2}$ is contained in the present, as in Figure~\ref{fig:realmap-Dax}(ii).

    Use the normal ball to $e$ at $p$ to construct a regular nullhomotopy $F$, and compute the Dax invariant $\Dax(F)$. \emph{Hint.} Figure~\ref{fig:realmap-Dax}(iii).
\end{exercise}

\section{Computing the fundamental group of embeddings}\label{sec:fund-gp}

Next, we must address the problem of indeterminacy of the Dax invariant: if $F'$ is another regular nullhomotopy for $f$, do we have $\Dax(F')=\Dax(F)$? The answer is negative in general, because concatenating these maps along $f$ gives a regular nullhomotopy $FF'$ of $\const_e$, which can in general have double points.

We view the nullhomotopy $FF'$ as a map
\[
    [0,1]\tm[0,1]\ra\Immp(C,M)
\]
that is constantly equal to $e$ on the boundary:
\begin{equation}\label{eq:commut-schema}
\begin{tikzcd}
      e\rar[equals]\dar[equals]\ar[phantom]{dr}{F} 
      & e\dar[equals] \\
      e\ar[r,"f" description]\dar[equals]\ar[phantom]{dr}[rotate=180]{F'} 
      & e\dar[equals]\\
      e\rar[equals]
      & e
\end{tikzcd}
\end{equation}
Thus this represents a class $[FF']\in\pi_2(\Immp(C,M),e)$, and in Corollaries~\ref{cor:imarcs} and~\ref{cor:imcircles} we computed
\begin{align*}
    \pi_2(\Immp(\D^1,\X),\unkn) &\cong \Z\oplus\pi_3\X,\\
    \pi_2(\Imm(\S^1,N),\cknot) &\cong \Big(\faktor{\pi_3\S TN}{b\sim[\deriv\cknot]\cdot b}\Big).\Fix_{[\cknot]}(\pi_2N).
\end{align*}
    The class $1\in\Z$ corresponds to the positive cusp nullhomotopy $F_1$ of $\unkn\colon\D^1\hra\X$, and one can compute $\Dax(F_1)=1$; see \cite{KT-4dLBT}. This is the reason why we discard the group element $1$ from the group ring $\Z[\pi_1\X]$, and for the case of arcs need only compute the indeterminacies coming from $\pi_3\X$, as follows.

\subsection{The Dax homomorphism}\label{subsec:dax-homomorphism}

Let us first consider the case of arcs, that is, $(e,C,M)=(\unkn,\D^1,X)$. Given $a\in\pi_3X$ we can pick $A_\unkn\colon\S^2\to\Imm_\partial(\D^1,X)$ for which $\pi_2p_{\X,\unkn}(A_\unkn)=a$ and view it as a self-nullhomotopy of the loop $\const_\unkn$. See Figure~\ref{fig:u-whiskered}. We can compute $\dax_\unkn(a)\coloneqq\Dax(A_\unkn)$, and thus define the homomorphism
\begin{equation}\label{eq:dax-u-def}
    \dax_\unkn\colon\pi_3X\ra\Z[\pi_1X\setminus1].
\end{equation}
The image of this homomorphism is precisely the indeterminacy subgroup by which we must quotient.
\begin{figure}[!htbp]
    \centering
    \includegraphics[width=0.42\linewidth]{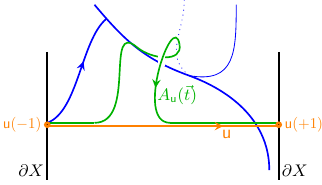}
    \captionof{figure}{A representative of $a\in\pi_3\X$ is colored blue, and the immersed arc $A_\unkn(\vec{t})$ for a value $\vec{t}\in\S^2$ is green.}
    \label{fig:u-whiskered}
\end{figure}

The following is an instructive computation of the Dax homomorphism; something similar appears in \cite[Lem.3.22,Fig.3.23]{KT-4dLBT} and \cite[Cor.1.4]{K-Dax}. See Corollary~\ref{cor:X-from-N} for an application.
\begin{lemma}\label{lem:Phi}
    Fix a 4-manifold $N$, an embedded circle $\cknot\colon\S^1\hra N$, and a small 4-ball $\D^4$ around the basepoint $\cknot(-1)$. Let $\Phi\colon\S^3\hra N\setminus\D^4$ be the boundary of the removed 4-ball, and let $\unkn\colon\D^1\hra N\setminus\D^4$ be the restriction of $\cknot$. Then for any $g\in\pi_1N$ we have
    \[
        \dax_{\unkn}(g[\Phi]) = g + g^{-1}[\cknot]^{-1}.
    \]
\end{lemma}
\begin{figure}[!htbp]
    \centering
    \includegraphics[width=0.8\linewidth]{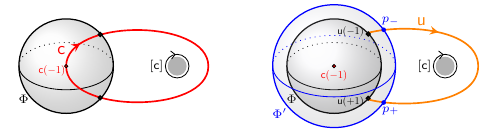}
    \caption{
    (i) Removing a $4$-ball from a manifold $N$ leaves a boundary component $\Phi$ and turns an embedded circle $\cknot$ into a neatly embedded arc $\unkn$. 
    (ii) The push-off $\Phi'$ is a 3-sphere that intersects $\unkn$ in points $p_-$ and $p_+$.}
    \label{fig:Phi}
\end{figure}
\begin{proof}
    The setup is depicted in Figure~\ref{fig:Phi}. We push $\Phi$ into the interior of $N$ to create an embedded sphere $\Phi'\colon\S^3\hra N\setminus\D^4$, that intersects $\unkn$ in two points $p_{\pm}$ of necessarily opposite signs.

    We now take a piece of $\unkn$ and bring it close to $\Phi'$, by first taking it around a loop representing $g\in\pi_1N$. We then let this piece move around $\Phi'$ as a 2-parameter family of arcs $A_\unkn(\vec{t})$ with fixed endpoints, $\vec{t}\in\S^2$. In Figure~\ref{fig:g+g-inv} on the left we see the basepoint arc $\unkn=A_\unkn(*)$, whereas on the right is the only immersed arc $A_\unkn(\vec{t}_0)$ in the family. We parametrize the family so that in the $s$-direction (for the loop) the green arc travels positively through time, whereas in the $\tau$-direction (for the homotopy) it gets pushed towards the reader, out of the page, creating double points.
\begin{figure}[!htbp]
    \centering
    \includegraphics[width=0.9\linewidth]{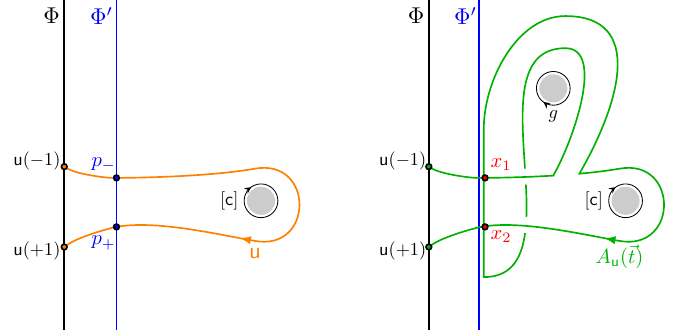}
    \caption{
    Foliating $g\Phi'$ by $\unkn$-based arcs yields a 2-parameter family $A_\unkn(\vec{t})$.
    (i)~The basepoint arc $A_\unkn(*)=\unkn$.
    (ii)~The immersed arc $A_\unkn(\vec{t}_0)$ with two double points.}
    \label{fig:g+g-inv}
\end{figure}

    As computed in Figure~\ref{fig:compute}, the double point loop is $g$ at $x_1$ and $g^{-1}[\cknot]^{-1}$ at $x_2$. 
\begin{figure}[!htbp]
    \centering
    \includegraphics[width=0.7\linewidth]{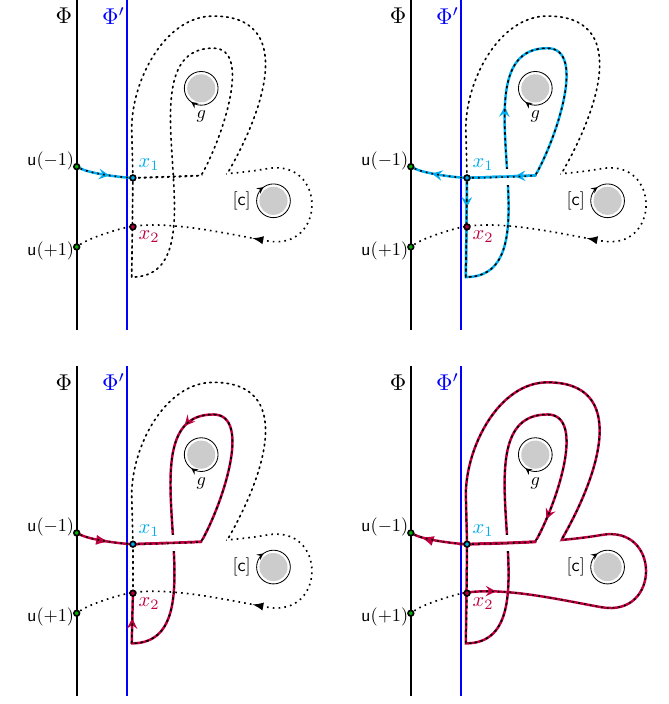}
    \caption{The computation of group elements. In the first row we see in blue respectively the curves $\gamma_1$ and $\gamma_2$ for $x_1$. In the second row we see in red respectively the curves $\gamma_1$ and $\gamma_2$ for $x_2$.}
    \label{fig:compute}
\end{figure}
    
    Using the Algorithm we see that both signs are positive, as follows. For $x_1$ the first sheet stands still, and the second sheet moves, so $\e_1=-1$, whereas the basis from the Step (v) of the Algorithm is ``(out of the page, up, right)'', which is negatively oriented, so $\e_2=-1$. On the other hand, at $x_2$ we have $\e_1=+1$ as the first sheet moves, and $\e_2=+1$ because the basis is ``(out of the page, up, left)''.
\end{proof}

\subsection{Relating Dax and Wall}\label{subsec:Dax-Wall}

As mentioned in §\ref{subsec:wall} and discussed in more detail in Appendix~\ref{app:WHD}, the results of Dax generalize those of Wall and Haefliger. The following theorem of the author and Peter Teichner is another comparison of that kind, but particular to the setting of 1-manifolds in 4-manifolds. We compare the Dax homomorphism $\dax_\unkn$ from §\ref{subsec:dax-homomorphism} --- in the special case when $\unkn$ is homotopic into the boundary --- to certain variants of the Wall invariants $\mu$ from §\ref{subsec:wall}.

Namely, in \cite[Cor.4.13]{KT-4dLBT} we show that if $\unkn$ is homotopic into the boundary, then the image of $\dax_\unkn$ is contained in the subgroup
\[
    \Z[\pi_1\X]^\sigma<\Z[\pi_1\X]
\]
of invariants under the involution $g\mapsto g^{-1}$.
We then consider the Wall invariants
\[
    \mu_3\colon H_3\wt\X \ra \faktor{\Z[\pi_1\X]^\sigma}{1\sim0,\, g\sim-g^{-1}},
\]
which counts self-intersections of immersed 3-manifolds in the 6-manifold $\X\tm\R^2$, and 
\[
    \mu_2\colon\pi_2\X\ra \faktor{\Z[\pi_1\X]}{1\sim0,\, g\sim g^{-1}},
\]
which counts self-intersections of immersions of 2-spheres into the 4-manifold $\X$.

\begin{theorem}[{K.--Teichner~\cite[Thm.3.17]{KT-4dLBT}}]\label{thm:Hopf}
  Fix a 4-manifold $\X$ and an arc $\unkn\colon\D^1\hra\X$ that is homotopic rel endpoints into the boundary $\partial\X$. Then there is a commutative diagram of 
  \[\begin{tikzcd}
       \Gamma(\pi_2\X)
            \ar[tail]{rr}{\Gamma(-\circ H)}
            \dar{\Gamma(\mu_2)}
        && \pi_3\X 
            \ar[two heads]{r}{\mathrm{Hur}} 
            \dar{\dax_\unkn} 
        & H_3\wt\X 
            \dar{\mu_3} 
        \\
        \faktor{\Z[\pi_1\X]}{1\sim0,\, g\sim g^{-1}}
            \ar[tail]{rr}{g\ \mapsto\ g+g^{-1}} 
        && \Z[\pi_1\X\setminus1]^\sigma 
            \rar[two heads]
        & \faktor{\Z[\pi_1\X]^\sigma}{1\sim0,\, g\sim-g^{-1}}
    \end{tikzcd}
    \]
    In particular, $\dax_\unkn(a\circ H) = \mu_2(a) +\overline{\mu_2(a)}$ for all $a\in\pi_2X$. 
\end{theorem}
The horizontal maps on the right are the Hurewicz map for the universal cover $\wt X$ and the natural quotient map.
The bottom row is exact by construction, whereas the top row is Whitehead's ``certain exact sequence'' \cite{Whitehead-Certain-Exact-Seq}, with the homomorphism out of $\Gamma(\pi_2X)$ determined by the quadratic map given by precomposition with the Hopf map $H\colon\S^3\to \S^2$. It is injective as the preceding term in Whitehead's sequence is $H_4\wt X$ which vanishes since $\partial X\neq\emptyset$.

\subsection{The Dax-whisk homomorphism}\label{subsec:dax-whisk}

The situation with indeterminacies is slightly more subtle in the case of circles $(e,C,M)=(\cknot,\S^1,N)$, because of the presence of the group extension describing $\pi_2$ of immersions. Using an indirect argument that goes through the case of arcs, in \cite{K-Dax} we showed that the indeterminacy group in this case is the sum of the corresponding group
\[
    \dax_\unkn(\pi_3(N\setminus\D^4))
\]
for the case of arcs based at $\unkn\colon\D^1\hra N\setminus\D^4$ --- obtained by removing a small ball near the basepoint of $\cknot$ --- and the group
\[
    \dax^{whisk}_\cknot(\Fix_{[\cknot]}(\pi_2N))
\]
that measures indeterminacy of the choice of the basepoint. The homomorphism
\begin{equation}\label{eq:dax-whisk-def}
    \dax^{whisk}_\cknot\colon \Fix_{[\cknot]}(\pi_2N) \ra \Z[\pi_1N\setminus1]
\end{equation}
    is computed as
\[
    \dax^{whisk}_\cknot(b)\coloneqq\Dax(H_b)
\]
for the regular nullhomotopy $H_b$ defined as follows. 
    
    First represent $b\in\pi_2N$ satisfying $b=\cknot\cdot b$ by a map 
    \[
    b\colon [0,1] \to \Omega N=\Map((\S^1, -1), (N, *))
    \]
    with $b(0) = b(1) = \const_*$. Then use ambient isotopy extension to lift this to a family $B\colon [0, 1] \to \Omega \Emb(\S^1, N)$ such that $B(1) = \const_\cknot$ and $B(t)(-1) = b(t)$. 
    
    Observe that the class of $B(0)|_{\S^1\setminus nbhd(-1)}$ in $\pi_1(\Embp(\D^1, N\setminus\D^4); \unkn)$ is trivial when included in $\pi_1(\Emb(\S^1, N); \cknot)$. What is more, the equality $b=\cknot\cdot b$ means that there is also a homotopy from the loop $B(0)|_{\S^1\setminus nbhd(-1)}$ of arcs to $\const_\unkn$ through immersed neat arcs; this is our desired regular nullhomotopy $H_b$. We refer to \cite[Thm.4.5]{K-Dax} for details.

\begin{example}\label{ex:nullhtpic2}
    If $\cknot$ is nullhomotopic, the homomorphism 
    $\dax^{whisk}_\cknot\colon\pi_2N\to\Z[\pi_1N\setminus1]$ is trivial.
    Indeed, we can assume that $\cknot$ is contained in a small ball, and the ambient isotopy extension $B$ can be constructed to have this ball in its support. This gives the required $H_b$ that is now an isotopy, so $\Dax(H_b)=0$.
\end{example}

\subsection{The Dax theorem}

\begin{theorem}[Dax~\cite{Dax}, Gabai~\cite{Gabai-disks}, K.--Teichner~\cite{KT-highd}, K.~\cite{K-Dax}]
\label{thm:Dax-dax}
    The Dax invariant from Algorithm~\ref{alg} induces isomorphisms
    \[\begin{tikzcd}[row sep=tiny]
        \Dax\colon\; \ker(\pi_1 p_{\X,\unkn})
        \rar[tail,two heads]{} &
        \faktor{\Z[\pi_1\X\setminus1]}{\dax_{\unkn}(\pi_3\X)},\\
        \Dax\colon\; \ker(\pi_1 p_{N,\cknot})
        \rar[tail,two heads]{} &
        \faktor{\Z[\pi_1N\setminus1]}{\dax_{\unkn}(\pi_3(N\setminus\D^4))
    \oplus\dax^{whisk}_\cknot(\Fix_{[\cknot]}(\pi_2N))}.
    \end{tikzcd}
    \]
    Moreover, these groups are generated by the families $\prealmap(g)$ from Figure~\ref{fig:realmap-Dax}, for $g\in\pi_1\X$ or $g\in\pi_1N$ respectively.
\end{theorem}

For the last sentence, recall that in Exercise~\ref{exerc:realmap} we have described nullhomotopies $H_g$ for the loops $\prealmap(g)$, with $\Dax(H_g)=g$.

\begin{corollary}\label{cor:main}
	There are group extensions
\[
    \begin{tikzcd}[scale cd=0.9,row sep=tiny]
        \faktor{\Z[\pi_1\X\setminus1]}{\dax_{\unkn}(\pi_3\X)}\rar[tail]{\prealmap} 
        & \pi_1(\Embp(\D^1,\X), \unkn)\rar[two heads]{\pi_1 p_{\X,\unkn}} 
        & \pi_2\X\\
        \faktor{\Z[\pi_1N\setminus1]}{rel_{\cknot}}\rar[tail]{\prealmap} 
        & \pi_1(\Emb(\S^1,N), \cknot)\rar[two heads]{\pi_1 p_{N,\cknot}} 
        & \Big(\faktor{\pi_2N}{b\sim[\cknot]\cdot b}\Big).\Big(\Fix_{[\cknot]}(\pi_1N)\Big)
        \end{tikzcd}
\]
    for $rel_{\cknot}=\dax_{\unkn}(\pi_3(N\setminus\D^4))
    \oplus\dax^{whisk}_\cknot(\Fix_{[\cknot]}(\pi_2N))$.
\end{corollary}
One can show that these are central extensions and express the commutator pairing~\cite[Prop.3.11]{KT-4dLBT}. In particular, the fundamental groups of these embedding spaces are rarely abelian.

\begin{remark}\label{rem:circles-from-arcs}
	Given $(\cknot,\S^1,N)$ define $(\unkn,\D^1,X)\coloneqq (\cknot\setminus(\cknot\cap\D^4), \D^1, N\setminus\D^4)$. Then the kernel of $\pi_1 p_{N,\cknot}$ can be described as the quotient of the kernel of $\pi_1 p_{\X,\unkn}$ by the subgroup $\dax^{whisk}_\cknot(\Fix_{[\cknot]}(\pi_2N))$.
\end{remark}

\section{Examples}\label{sec:examples}

Throughout this section $\X$ and $N$ are smooth $4$-manifolds, and we fix basepoint embeddings $\unkn\colon\D^1\hra\X$ and $\cknot\colon\S^1\hra N$.

\subsection{Arcs}

The following two results follow immediately from Corollary~\ref{cor:main}.
\begin{corollary}\label{cor:1-trivial}
    If $\X$ is simply connected then $\pi_0\Arcs{\X}=\{\unkn\}$ and embeddings versus immersions cannot be distinguished by $1$-parameter families. In fact, such families precisely correspond to second homology classes: 
    \begin{align*}
        \pi_1(\Arcs{\X},\unkn)&\cong\pi_1(\ImArcs{\X},\unkn)\cong\pi_2\X\cong H_2(\X;\Z).
    \end{align*}
\end{corollary}

The following is another extreme situation.
\begin{corollary}
    If $\pi_2\X=\pi_3\X=0$, then
    \[\pi_1(\Embp(\D^1,\X),\unkn)\cong \Z[\pi_1\X\setminus1].
    \]
\end{corollary}

Here are some examples that satisfy these conditions.

\begin{example}
We have
\[
    \pi_1(\Embp(\D^1,\S^1\tm\D^3),\unkn)
    \cong \Z[\Z\setminus1]
    \cong \faktor{\Z[t,t^{-1}]}{1\sim 0}
    \cong\Z^\infty,
\]
with generators $t, t^{-1}, t^2, t^{-2},\dots$.
\end{example}

\begin{example}
    Let $\X$ be a fiber bundle with fiber $\Sigma_g\coloneqq (\S^1\tm\S^1)^{\#g}$ over the surface $\Sigma_{h,1}\coloneqq \Sigma_h\setminus\D^2$, for some $g,h\geq1$. This space is determined by a homomorphism 
    \[
    \psi\colon\pi_1\Sigma_{h,1}\cong F_h\to \pi_0\Diff(\Sigma_g),
    \]
    since the base is homotopy equivalent to a wedge of circles. We thus write
    \[
    \X\cong \Sigma_g\rtimes_\psi\Sigma_{h,1}.
    \]
    The homotopy long exact sequence implies that $\pi_n\X=0$ for all $n\geq2$, and $\pi_1\Sigma_g\hra \pi_1\X\sra F_h$. This splits, and the action $\psi_*\colon F_h\to Aut(\pi_1\Sigma_g)$ is induced by $\psi$. We have
    \[
    \pi_1(\Embp(\D^1,\Sigma_g\rtimes_\psi\Sigma_{h,1}),\unkn)
    \cong \faktor{\Z[(\pi_1\Sigma_g)\rtimes_{\psi_*}F_h]}{1\sim 0}\cong  \faktor{\Z[\pi_1\Sigma_g]\rtimes_{\psi_*}F_h}{1\sim 0},
    \]
    the group ring of $\pi_1\Sigma_g$, twisted by the action $\psi_*$, modulo the trivial group element.
\end{example}

Another interesting family of cases is provided by punctured manifolds.
\begin{corollary}\label{cor:X-from-N}
    Fix a small 4-ball $\D^4\subset N$. Let $\unkn\colon\D^1\hra N\setminus\D^4$ be an arc with endpoints on $\Phi=\partial\D^4$, and join them on $\Phi$ to get a loop $[\cknot]\in\pi_1N$. Then 
    \[\begin{tikzcd}
    \faktor{\Z[\pi_1N]}{1\sim 0, g\sim -g^{-1}[\cknot]^{-1},\dax_{\unkn}(\pi_3N)}\rar[tail] 
    & \pi_1(\Emb(\D^1,N\setminus\D^4), \unkn)\rar[two heads]
    & \pi_2N.
    \end{tikzcd}
\]
\end{corollary}
Note that here we abuse notation: we have so far defined $\dax_\unkn$ only for classes in $\pi_3(N\setminus\D^4)$, because $\unkn$ lives in $N\setminus\D^4$. However, the proof of the corollary will show that after modding out by the relations $g\sim -g^{-1}[\cknot]^{-1}$ we can push any map $\S^3\to N$ off the ball $\D^4$, and compute its $\dax_\unkn$-invariant.

For the proof we will need the following fact, see for example~\cite[Lem.4.9]{K-Dax}.
\begin{lemma}\label{lem:rational-h-s}
    Assume $N$ is not a simply connected rational homology sphere. Then $\pi_1(N\setminus\D^4)\cong\pi_1N$ and $\pi_2(N\setminus\D^4)\cong\pi_2N$, and we have
    \[\begin{tikzcd}
        \Z[\pi_1N]\rar[tail] & \pi_3(N\setminus\D^4)\rar[two heads] & \pi_3N,
    \end{tikzcd}
    \]
    where the first map sends $g\in\pi_1N$ to $g[\Phi]$, for $\Phi\colon\S^3\hra N\setminus\D^4$.
\end{lemma}
\begin{proof}[Proof of Corollary~\ref{cor:X-from-N}]
    If $N$ is a simply connected rational homology sphere, then $\pi_1(\Emb(\D^1,N\setminus\D^4), \unkn)\cong\pi_2N$ by Corollary~\ref{cor:1-trivial}, so the claim holds.
    Otherwise, we are in the situation of Lemma~\ref{lem:rational-h-s}, so we have
    \[
    \faktor
    {\Z[\pi_1(N\setminus\D^4)\setminus1]}
    {\dax_{\unkn}(\pi_3(N\setminus\D^4))}
    \cong \faktor{\left(\faktor{\Z[\pi_1N\setminus1]}{\langle\dax_\unkn(g\Phi):g\in\pi_1N\rangle}\right)}{\dax_{\unkn}(\pi_3N)}.
    \]
    In other words, $\dax_{\unkn}$ can be defined on $\pi_3N$, once we quotient out $\dax_\unkn(\Z[\pi_1N])$. 
    
    Finally, by Lemma~\ref{lem:Phi} we compute $\dax_\unkn(g[\Phi])= g + g^{-1}[\cknot]^{-1}$, proving the desired statement.
\end{proof}

\begin{remark}
	Note that applying the relation $g\sim-g^{-1}[\cknot]^{-1}$ once again gives \[
    g^{-1}[\cknot]^{-1} \sim -(g^{-1}[\cknot]^{-1})^{-1}[\cknot]^{-1}.
    \] 
    This implies that for any $g\in\pi_1N$ the relation $g\sim [\cknot] g [\cknot]^{-1}$ will hold in the quotient.
\end{remark}

Here is an example that uses Corollary~\ref{cor:X-from-N}.

\begin{example}\label{ex:intro-arcs}
	Consider the manifold $N=\S^1\tm\D^3\#\S^1\tm\S^3$ with  $\cknot=\S^1\tm p_{\S^{3}}\subset N$  and $G\coloneqq [p_{\S^{1}}\tm\S^{3}]$. Let $\unkn\colon \D^1\hra N\setminus\D^4$ be the arc obtained from $\cknot$ by removing the $4$-ball. Then
	\[
	\pi_1(\Embp(\D^1,\S^1\tm\D^3\#(\S^1\tm\S^3\setminus\D^4)), \unkn) 
	\cong 
	\faktor{\Z[\mathbb{F}_2]}{1\sim0, g\sim -g^{-1}t^{-1},\dax_{\unkn}(gG)}.
	\]
	Indeed, this follows from Corollary~\ref{cor:X-from-N} using the observation that
\[
	N\cong \S^1\tm\D^3\natural (\S^1\tm\S^3\setminus\D^{4}) \simeq \S^{1}\vee\S^{1}\vee\S^{3}
\]
	to compute $\pi_1N\cong\mathbb{F}_{2}$, generated by $s=[\S^1\tm\{p_{\D^{3}}\}]$ and $t=[\S^1\tm\{p_{\S^{3}}\}]=[\cknot]$, and $\pi_{2}N=0$, and $\pi_{3}N\cong\Z[\mathbb{F}_{2}]$, generated by $G=[p_{\S^{1}}\tm\S^{3}]$.
\end{example}

One can compute the relations $\dax_{\unkn}(gG)$.
\begin{exercise}\label{ex:intro-computation}
	In the setting of Example~\ref{ex:intro-arcs} for $g\in\mathbb{F}_{2}=\langle s\rangle\ast\langle t\rangle$, show that
	\[
	\dax_{\unkn}(gG)=\begin{cases}
		g, & \text{if } g=s^k,\\
		g-gH_g+\ol{H_g}g^{-1}, & \text{if } g=s^k\prod_{i=1}^n t^{\e_{i}}s^{k_{i}} \text{ for } n\geq0,k_{i}\in\Z,\e_{i}\in\{\pm1\},
	\end{cases}
	\]
where $H_g\coloneqq \sum_{j=1}^{n}\e_{j}(s^k\prod_{i=1}^{j-1} t^{\e_{i}}s^{k_{i}})^{-1}$.
\emph{Hint.} Use \cite[Cor.B(v)]{K-Dax}.
\end{exercise}

\begin{question}
	How can one simplify this presentation of the abelian group $\pi_1(\Embp(\D^1,\S^1\tm\D^3\#\S^1\tm\S^3\setminus\D^4), \unkn)$?
\end{question}

\subsection{Circles}

\begin{corollary}\label{cor:1-trivial-circ}
    If $N$ is simply connected then $\pi_0\Emb(\S^1,N)=\{\cknot\}$ and embeddings versus immersions cannot be distinguished by $1$-parameter families. In fact, such families precisely correspond to second homology classes: 
    \begin{align*}       
    \pi_1(\Emb(\S^1,N),\cknot)&\cong\pi_1(\Imm(\S^1,N),\cknot)\cong \pi_2N\cong H_2(N;\Z).
    \end{align*}
\end{corollary}
\begin{proof}
    This is immediate from Corollary~\ref{cor:main}.
\end{proof}

\begin{corollary}
    If $\cknot$ is nullhomotopic, then 
\[\begin{tikzcd}
        \faktor{\Z[\pi_1N]}{1\sim 0, g\sim -g^{-1},\dax_{\unkn}(\pi_3N)}
        \rar[tail] 
        & \pi_1(\Emb(\S^1,N), \cknot)\rar[two heads]
        & \pi_2N\rtimes\pi_1N.
        \end{tikzcd}
\]
\end{corollary}
\begin{proof}
    This follows from Corollary~\ref{cor:main}, Exercise~\ref{ex:nullhtpic} and Example~\ref{ex:nullhtpic2}.
\end{proof}

\begin{remark}
    If $\cknot$ is nullhomotopic, then $\unkn\colon\D^1\hra N\setminus\D^4$ is homotopic into the boundary, and computations of $\dax_\unkn$ become somewhat simpler; see~\cite{K-Dax}.
\end{remark}

\begin{corollary}\label{cor:pi2-pi3-0}
     If $\pi_2N=\pi_3N=0$, then
\[\begin{tikzcd}
        \faktor{\Z[\pi_1N\setminus1]}{g\sim -g^{-1}[\cknot]^{-1}}\rar[tail] & \pi_1(\Emb(\S^1,N), \cknot)\rar[two heads]
        & \Fix_{[\cknot]}(\pi_1N).
    \end{tikzcd}
\]
\end{corollary}

\begin{proof}
    This follows from Corollaries~\ref{cor:main} and~\ref{cor:X-from-N}.
\end{proof}

\begin{example}\label{ex:S1xD3}
    As a particular example of Corollary~\ref{cor:pi2-pi3-0}, we find
    \[
    \begin{tikzcd}
        \faktor{\Z[\Z\setminus1]}{g\sim -g^{-1}t^{-W_0}}\rar[tail]
        & \pi_1(\Emb(\S^1,\S^1\tm\D^3), t^{W_0})\rar[two heads]
        & \Z.
    \end{tikzcd}
    \]
    Here we denote a generator of $\pi_1(\S^1\tm\D^3)\cong\Z$ by $t$, and the ``winding number'' $W_0\in\Z$ is fixed.
    We claim that this extension is trivial, so that:
        \begin{align*}
        \pi_1(\Emb(\S^1,\S^1\tm\D^3), t^{W_0})
        & \cong \Z\oplus\faktor{\Z[\Z\setminus1]}{g\sim -g^{-1}t^{-W_0}}\\
        & \cong \Z\oplus\faktor{\Z[t,t^{-1}]}{1\sim 0, t^{-k}\sim -t^{k-W_0}}
    \end{align*}
    For this we first observe that the extension has a splitting: send $n\in\Z$ to the class $f\in \pi_1(\Emb(\S^1,\S^1\tm\D^3), t^{W_0})$ induced from $n$ full rotations of the ambient space $\S^1\tm\D^3$. Moreover, this lands in the center, as we can again rotate the ambient space while doing any other class. Thus, this semi-direct product is the product.
\end{example}

A related example is the following; see~\cite{Budney-Gabai} and~\cite[Cor.2.6]{K-Dax}.

\begin{example}
    We have
    \[
    \pi_1(\Emb(\S^1,\S^1\tm\S^3), t^{W_0})
        \cong \Z\oplus\faktor{\Z[t,t^{-1}]}{1\sim 0, t^{-k}\sim -t^{k-W_0}, t^{-1}+\cdots+t^{-W_0}\sim0}
    \]
    The only difference to the last example is that now we have $\pi_3N\cong\Z$. A generator is represented by the inclusion $i_2\colon\{pt\}\tm\S^3\hra \S^1\tm\S^3$ and we claim
    \[
        \dax_\unkn(i_2)=t^{-1}+\cdots+t^{-W_0}.
    \]
    Here the basepoint arc is $\unkn\colon\D^1\hra N\setminus\D^4$, obtained from $\cknot=t^{W_0}$ by removing its intersection with the ball $\D^4$. We take a piece of $\unkn$ and drag it around the sphere $i_2$ in a 2-parameter family, and look for double points that occur. Since $i_2$ is embedded, we will see only double points that arise from intersections of some other part of $\unkn$ and $i_2$. In other words, we are computing the equivariant intersection number $\lambda(\unkn,i_2)$. Since the winding number is $W_0$, we see precisely $|W_0|$ many intersections, and that the double point loops are $t^{-n}$ for $n=1,\dots,W_0$.
\end{example}

\begin{example}\label{ex:intro-circles}
	Here is the promised example from the introduction. The kernel of the surjection
\[\begin{tikzcd}[column sep=small]
        \pi_1\big(\Emb(\S^1,\S^1\tm\D^3\#\S^1\tm\S^3), \S^1\tm\{p_{\S^{3}}\}\big)
        \rar[two heads] 
        & \Fix_{[\cknot]}(\pi_1N)\cong\Z
\end{tikzcd}
\]
agrees with the group from Example~\ref{ex:intro-arcs}. Namely, it is the quotient of the group ring of the free group on two generators $s,t$ by the relations 
\[
	1\sim0,\quad 
	g\sim-g^{-1}t^{-1},\quad 
	s^k\sim0,\quad 
	w \sim \sum_{j=1}^{n}\e_{j} \Big(wv_j^{-1}-v_j w^{-1}\Big),
\]
for $g\in\mathbb{F}_{2}$, $k\in\Z$, and $w=s^k\prod_{i=1}^n t^{\e_{i}}s^{k_{i}}$ for $n\geq1$, $\e_i\in\{\pm1\}$, $k_i\in\Z$, with
\[
	v_j\coloneqq
	s^k\prod_{i=1}^{j-1} t^{\e_{i}}s^{k_{i}}.
\]
This follows from Remark~\ref{rem:circles-from-arcs}, Example~\ref{ex:intro-arcs} and Exercise~\ref{ex:intro-computation}.
\end{example}

\section{Beyond the fundamental group}\label{sec:beyond}

\subsection{Graspers}\label{subsec:graspers}

The operation from classical knot theory presented in Example~\ref{ex:intro}(1), that produces the trefoil from the unknot $\unkn$, is called the Gusarov--Habiro {clasper surgery} on $\unkn$ of degree $n=2$ \cite{Gusarov-main,Gusarov,Habiro}. For $n=1$ clasper surgery is simply a crossing change: ambiently connect-sum a subarc of $\unkn$ into a meridian of $\unkn$. In general, clasper surgery of degree $n\geq1$ uses $n-1$ copies of the Borromean rings and $n+1$ subarcs of $\unkn$, and depends on the information of a planar rooted binary tree with $n$ leaves. As in that example, this corresponds to taking iterated commutators in the fundamental group of the complement of a knot. 

Another viewpoint on this operation uses {gropes} of Conant and Teichner~\cite{CT1}. A grope contains a clasper as a subset, but also carries additional data; this is related to trivializations of the iterated commutator, once one of the letters is replaced by the unit.

Claspers and gropes can be generalized to all embedding spaces for which the dimension of the source is one and of the target is at least three, using the following notion of graspers~\cite{K-families}. Here we concentrate on $C=\D^1$ or $C=\S^1$ (both have basepoint $-1$) and $\dim M=4$.
\begin{enumerate}
    \item
A \emph{grasper} of degree $n$ relative to $e\colon C\hra M$ is an embedding $\TG_n\colon\ball^4\hra M$, which is disjoint from $\partial M$, and which intersects the basepoint knot $e$ only in the fixed intervals
\begin{equation}\label{eq:G-condition}
    \TG_n(\ball^4)\cap e(C)=\bigsqcup_{i=0}^ne(J_i)=\bigsqcup_{i=0}^n\TG_n(a_i),
\end{equation}
for some fixed arcs $a_i\colon\D^1\hra\ball^4$ with $0\leq i\leq n$. See Figure~\ref{fig:grasper}. 

For a grasper $\TG_n\colon\ball^4\hra M$ and $0\leq i\leq n$ let $p_i\in C$ be the point such that $e(p_i)=\TG_n(a_i(0))$. We define $g_i(\TG_n)\in\pi_1(M,e(-1))$
as the homotopy class of the loop given by $e|_{[0,p_i]}$ followed by a path from $e(p_i)$ to $e(p_0)$ inside of $\TG_n(\ball^4)$, then followed by $e|_{[0,p_0]}^{-1}$. We say that the tuple
\[
(g_1(\TG_n),\dots,g_n(\TG_n))\in(\pi_1M)^n
\]
\emph{decorates} the grasper $\TG_n$.
\begin{figure}[!htbp]
    \centering
    \includegraphics[width=0.95\linewidth]{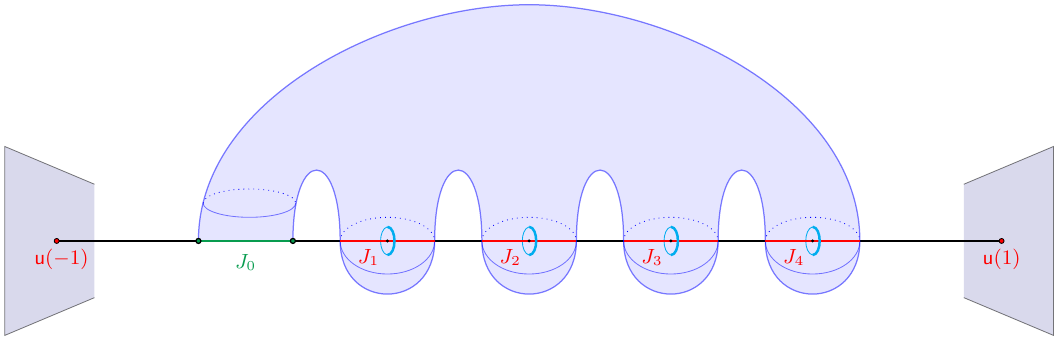}
    \caption{A grasper in a $4$-manifold $M$ relative to $e\colon\D^1\hra M$ (the horizontal line) is a $4$-ball $\TG_n$ that ``grasps'' $e$ in fixed subintervals $e(J_i)=\TG_n(a_i)$, for $0\leq i\leq n$. Here $n=4$. 
    }
    \label{fig:grasper}
\end{figure}

    \item 
Fix a rooted binary planar tree $\Gamma$ with $n$ leaves. We regard it as embedded in $\ball^4$, with the root on $a_0$ and leaves on $a_i$, as in Figure~\ref{fig:tree}. One can define a certain ``universal $\Gamma$-family''
\begin{equation}\label{eq-intro:mu}
    \mu_{\Gamma}\colon \big( \S^{n}, * \big) \ra
    \big( \Embp(\D^1,\ball^4\setminus\bigsqcup_{i=1}^na_i), a_0 \big),
\end{equation}
using iterated commutators of the loops that foliate the meridian spheres $m_i(\S^2)$ of the arc $e(J_i)$, $1\leq i\leq n$. In analogy to gropes, this also comes with additional data: the family
\begin{equation}\label{eq-intro:mu-star}
    \mu^\star_{\Gamma}\colon \I^n\times\S^{n}\to \Embp(\D^1,\ball^4),
\end{equation}
that in certain sense trivializes $\mu_{\Gamma}$. This will be used in \eqref{eq:later}.
\end{enumerate}
\begin{figure}[!htbp]
    \centering
    \includegraphics[width=0.3\linewidth]{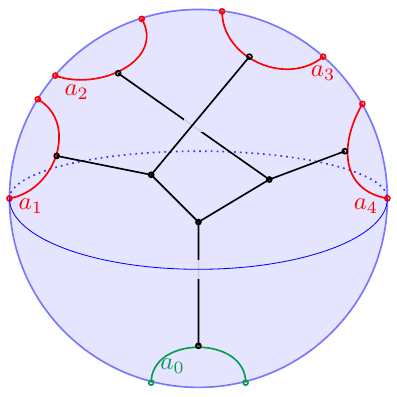}
    \caption{A tree $\Gamma$ in $\ball^4$ with four leaves. It corresponds to the word $[[a_4,a_2],[a_3,a_1]]$.
    }
    \label{fig:tree}
\end{figure}

\subsection{The realization map}\label{subsec:realmap}
Given a grasper $\TG_n\colon\ball^4\hra M$ and a tree $\Gamma$ with $n$ leaves, we define the \emph{grasper family}
\[
    e_{\TG_n,\Gamma}(\vec{\tau})\colon C\hra M,\quad \vec{\tau}\in\S^{n},
\]
by replacing the subarc $e(J_0)=\TG_n(a_0)$ with $\TG_n(\mu_{\Gamma}(\vec{\tau}))$. 

More precisely, the map $e_{\TG_n,\Gamma}$ is defined as the composite
\[\begin{tikzcd}
        \S^{n}\rar{\mu_{\Gamma}} & 
        \Embp(\D^1,\ball^4\setminus\bigsqcup_{i=1}^na_i)\rar{\TG_n\circ-} & 
        \Emb(\D^1,M\setminus\TG_n(\bigsqcup_{i=1}^na_i))\rar{e|_{C\setminus J_0}\cup-} &
        \Embp(C,M).
    \end{tikzcd}
\]

In particular, when $n=1$ we have only one tree $\Gamma=|$, and for a grasper $\TG_1\colon \ball^4\hra M$ we have a single group element $g_1(\TG_1)\in\pi_1M$. We can define a map 
    \[
    \realmap_1\colon\Z[\pi_1M]\ra\pi_1(\Embp(C,M),e)
    \]
to send $g\in\pi_1M$ to the homotopy class of $e_{\TG_1,|}\colon\S^1\to\Embp(C,M)$ for any grasper $\TG_1$ that is decorated by $g$, that is, has $g_1(\TG_1)=g$. In fact, this is precisely the family $\realmap(g)$ from Figure~\ref{fig:realmap-Dax}. The following is a consequence of Exercise~\ref{exerc:realmap}.

\begin{theorem}\label{thm:realmap}
    The homomorphism $\realmap_1$ defines an inverse for each of the two Dax invariants from Theorem~\ref{thm:Dax-dax}:
    \[\begin{tikzcd}
        \faktor{\Z[\pi_1M]}{rel}\rar[shift left=4pt]{\prealmap_1}[swap]{\cong}
        & \ker(\pi_1p)\lar[shift left=4pt]{\Dax}.
    \end{tikzcd}
    \]
\end{theorem}

More generally, we can define a homomorphism 
    \begin{equation}\label{eq-def:realmap}
    \realmap_n\colon
    \Z[\Tree(n)\otm(\pi_1M)^n]\ra\pi_{n}\Embp(C,M)
    \end{equation}
by linearly extending the map that sends a decorated rooted binary planar tree 
\[
	(\Gamma,g_1,\dots,g_n)\in\Tree(n)\otm(\pi_1M)^n
\]
 to the homotopy class $[e_{\TG_n,\Gamma}]$, for any grasper $\TG_n$ of degree $n$ in $M$ relative to $e$ and decorated by $(g_1,\dots,g_n)$.
However, in order to generalize Theorem~\ref{thm:realmap} we need to first explain how one generalizes Dax's results.

\subsection{Embedding calculus}\label{subsec:emb-calc}

As explained in Appendix~\ref{app:Dax} the range in which Dax's result applies is the range in which a map
    \[
        F\colon \Big(I^n,\, I^{n-1}\tm\{0\},\, \partial I^{n-1}\tm I\cup I^{n-1}\tm \{1\}\Big) \ra
        \Big( \Immp(C,M),\, \Embp(C,M),\, e \Big)
    \]
has the track
    \[
        \wt{F}\colon I^n\tm C\ra I^n\tm M,\quad \wt{F}(v,x)=(v,F(v)(x))
    \]
with only double points, and no triple points. The Dax invariant exactly counts these double points. This means that it suffices to consider what $\wt{F}$ does on configurations of two points. To generalize this for higher degrees $n$, however, one must consider configurations of increasingly many points.
The result is a tower of spaces
\begin{equation}\label{eq-def:tower}
\begin{tikzcd}[row sep=small]
     \cdots\rar & T_{n+1}(C,M)\rar{p_{n+1}} & T_n(C,M)\rar & \cdots\rar & T_2(C,M)\rar & T_1(C,M)\\
     && \Embp(C,M)\ar{ul}{\ev_{n+1}}\uar{\ev_n}\ar[bend right]{urr}{\ev_2}\ar[bend right]{urrr}{\ev_1} &&&
\end{tikzcd}
\end{equation}
that has the following properties, due to Goodwillie, Klein, and Weiss~\cite{GW,GKW,GKmultiple}.
\begin{enumerate}
    \item There is a homotopy equivalence $T_1(C,M)\simeq\Immp(C,M)$ under which $\ev_1$ corresponds to the canonical inclusion.
    \item The map $\ev_{n+1}$ induces isomorphisms $\pi_k\Embp(C,M)\cong \pi_kT_{n+1}(C,M)$ for $k\leq n$.
    \item The layer $F_{n+1}\coloneqq\fib(p_{n+1}\colon T_{n+1}(C,M)\to T_n(C,M))$ can be described in terms of configuration spaces $\Conf_k(M)$ of $k\leq n+1$ points in $M$.
\end{enumerate}

In particular,
\begin{align*}
    \ker(\pi_1\Embp(\D^1,M)\sra\pi_1\Immp(\D^1,M))&\cong
    \ker(\pi_1T_2(\D^1,M)\sra\pi_1T_1(\D^1,M))\\
    &\cong\faktor{\pi_1F_2}{\delta(\pi_2T_1(\D^1,M))}
\end{align*}
and
\[
	F_2=\Omega^2\hofib\Big(\Conf_2(M)\xrightarrow{pr_1\tm pr_2}\Conf_1(M)\tm\Conf_1(M)\Big)\simeq\Omega^2\hofib(M\vee\S^3\xrightarrow{pr_M} M).
\]
One can show that $\pi_1F_2\cong\Z[\pi_1M]$,
and that $\delta(\pi_2T_1(\D^1,M))=\Z\oplus\dax(\pi_3M)$, recovering Theorem~\ref{thm:Dax-dax}.

More generally, we proceed as follows. Consider the diagram
     \[\begin{tikzcd}
    && \pi_nF_{n+1}\dar\\
    \Z[\Tree(n)\otm(\pi_1M)^n]\rar{\realmap_n}\ar[dashed]{rru}{\realmap^\star_n} &
    \pi_n\Embp(\D^1,M)\rar{\pi_n\ev_{n+1}}\ar{dr}[swap]{\pi_n\ev_n} &
    \pi_nT_{n+1}\dar{\pi_np_{n+1}}\\
    && \pi_nT_n
    \end{tikzcd}
    \]
where the realization map $\realmap_n$ is from~\eqref{eq-def:realmap} and the evaluation maps $\ev_n,\ev_{n+1}$ and the projection $\pi_np_{n+1}$ are from~\eqref{eq-def:tower}. The composite $\pi_n\ev_n\circ\realmap_n$ is trivial~\cite{K-families}, implying that $\pi_n\ev_{n+1}\circ\realmap_n$ lands in $\ker(\pi_np_{n+1})$. 
In fact, using the $n$-parameter nullhomotopies $\mu_{\Gamma}^\star$ of $\mu_\Gamma$ mentioned in~\eqref{eq-intro:mu-star}, we can directly construct a lift
\begin{equation}\label{eq:later}
    \realmap^\star_n\colon \Z[\Tree(n)\otm(\pi_1M)^n]\ra \pi_nF_{n+1}.
\end{equation}

\begin{theorem}[\cite{K-thesis-paper,K-families}]
    For every $n\geq1$ there is an isomorphism
    \[
    \chi_{n+1}\colon\pi_nF_{n+1}\xrightarrow{\cong} \Z[\Tree(n)\otm(\pi_1M)^n]
    \]
    such that $\chi_{n+1}\circ\realmap^\star_n=\Id$.
\end{theorem}

Let us give an application of this. We also rely on the computations in the homotopy spectral sequence for the embedding calculus tower, by Scannell--Sinha~\cite[the lower nonvanishing line in Table 2]{Scannell-Sinha} and \cite[Table 5]{Turchin-bialgebra}.

\begin{corollary}[{\cite[Cor.1.6]{K-families}}]\label{cor:Dd}
    We have isomorphisms
    \begin{align*}
        \pi_1\Embp(\D^1,\D^2)&\cong 1,\\[-7pt]
        \pi_{2}\Embp(\D^1,\D^2)&\cong\Z\quad \text{generated by the grasper family $\realmap_2(\tree)$}.
    \end{align*}  
    For any $n\geq3$ grasper families generate a subgroup
    \[\begin{tikzcd}
        {\mathcal{A}^{\mathsf{T,even}}_n\otimes\Q}\rar[hook ]{\realmap_n} & \pi_{n}\Embp(\D^1,\D^4)\otimes\Q,
    \end{tikzcd}
    \]
    that realizes all classes in the lower nonvanishing line of the Goodwillie--Weiss--Sinha spectral sequence.
    For $n\leq 6$ these inclusions hold integrally, and one has
    \[\mathcal{A}^{\mathsf{T,even}}_1=1,\,
    \mathcal{A}^{\mathsf{T,even}}_2=\Z,\,
    \mathcal{A}^{\mathsf{T,even}}_3=\Z,\, 
    \mathcal{A}^{\mathsf{T,even}}_4=1,\, 
    \mathcal{A}^{\mathsf{T,even}}_5=\Z^2,\, 
    \mathcal{A}^{\mathsf{T,even}}_6=\Z.
    \]
\end{corollary}

\appendix

\section{Why are immersions ``algebraic''?}\label{app:Smale}

We sketch the proof of Smale's equivalence~\eqref{eq:Smale}. The key ingredient~\cite{Smale} is that for a fixed $0<\e<1$ the restriction map
\[
    \res_{[-1,\e]}\colon\Imm_{\partial_-}([-1,\e],X)\ra \Imm_{\partial_-}([-\e,\e],X)
\]
is a Serre fibration: any $I^k\to \Imm_{\partial_-}([-\e,\e],X)$, together with a lift on $I^{k-1}\tm\{0\}\cup (\partial I^{k-1})\tm I$, lifts to $I^k\to \Imm_{\partial_-}([-1,\e],X)$.

The second ingredient is the commutative square formed by restriction maps
\[\begin{tikzcd}
    \Immp(\D^1,X)\rar\dar & \Imm_{\partial_+}([-\e,1],X)\dar\\
    \Imm_{\partial_-}([-1,\e],X)\rar & \Imm([-\e,\e],X)
\end{tikzcd}\]
This is a pullback square, which means that the top left corner is equivalent to the limit of the rest of the diagram, the \textit{pullback}, defined as
\[
    \Big\{
        (f_-,f_+)\in
        \Imm_{\partial_-}([-1,\e],X)\tm\Imm_{\partial_+}([-\e,1],X) 
        : \res_{[-\e,\e]}(f_-)=\res_{[-\e,\e]}(f_+)\Big\}.
\]
In fact, this space is homeomorphic to $\Immp(\D^1,X)$: such a pair $(f_-,f_+)$ glues to a smooth map $f_-\cup f_+\colon\D^1\to X$ that is clearly an immersion; and conversely, any $f\in\Immp(\D^1,X)$ gives such a pair.

Next, we need the fact that if in a pullback square one of the maps into the corner is a Serre fibration, then the square is a \emph{homotopy pullback} as well. This means that the canonical map from $\Immp(\D^1,X)$ to the homotopy limit of the rest of the diagram --- the \textit{homotopy pullback}, defined as
\begin{align*}
    \Big\{(f_-,f_+,\gamma)\in
        &\Imm_{\partial_-}([-1,\e],X)\tm\Imm_{\partial_+}([-\e,1],X)\tm\Map\big([0,1],\Imm([-\e,\e],X)\big) \\
        &: \res_{[-\e,\e]}(f_-)=\gamma(0),\;\res_{[-\e,\e]}(f_+)=\gamma(1)\Big\}
\end{align*}
--- is a (weak) homotopy equivalence. 

Finally, the advantage of using homotopy limits is in the fact that we can replace the square by a homotopy equivalent one, and the homotopy limit will also change only by a homotopy equivalence. In particular, observe that both $\Imm_{\partial_-}([-1,\e],X)$ and $\Imm_{\partial_+}([-\e,1],X)$ are contractible, as we can shrink the immersion to the only fixed end, whereas immersions of $[-\e,\e]$ shrink arbitrarily close to their value at $0$, where they are determined by their derivative. Thus, we have a homotopy equivalence
\[
    \Imm([-\e,\e],X)\simeq \S TX
\]
with the unit tangent bundle of $\X$. Alternatively, we can write our square as
\[\begin{tikzcd}
    \Map_\partial(\D^1,\S TX)\rar\dar 
    & \Map_{\partial_+}([0,1],\S TX)\dar{\res_0}\\
    \Map_{\partial_-}([-1,0],\S TX)\rar{\res_0} 
    & \S TX
\end{tikzcd}\]
since here also both anti-diagonal corners are contractible. Thus, we conclude that the claimed equivalence~\eqref{eq:Smale} holds:
\[
    \Immp(\D^1,X)\simeq\Map_\partial(\D^1,\S TX).
\]
In the years that followed, Smale and Hirsch generalized this result to all immersion spaces, with the only condition $\dim C<\dim M$, or $\dim C=\dim M$ but $C$ is open.

\section{Whitney, Haefliger, Dax}\label{app:WHD}

\subsection{Whitney}\label{app:Whitney}

In 1936 Whitney published his foundational work~\cite{Whitney-manifolds} on smooth manifold topology. In particular, he showed that any smooth $c$-dimensional manifold $C$ immerses in $\R^{2c}$ with transverse self-intersections. This means that the set of double points is $0$-dimensional and discrete. To improve this to an embedding in~\cite{Whitney-n-in-2n} he proceeded as follows.

If $C$ is closed and oriented, the set of double points of such an immersion $f\colon C\imra \R^{2c}$ is finite; Whitney takes the sum of (signed) double points to define the algebraic self-intersection number $I_f\in\Z$ for $c$ even and $I_f\in\Z/2\Z$ for $c$ odd (compare to §\ref{subsec:wall}). 
By a cusp homotopy one can introduce a double point of any sign, making $I_f$ vanish. Then one can pair up double points and with a bit of work also find suitable embedded Whitney disks, so that one Whitney move removes each cancelling pair. Therefore, any smooth $c$-manifold embeds in $\R^{2c}$.

\subsection{Haefliger}\label{app:Haefliger}

Whitney's result was generalized by  Haefliger~\cite{Haefliger-plongements}, in order to find embeddings into general target manifolds, under a condition on dimensions. Namely, ``by general position'' it follows that a generic immersion $f\colon C\imra M$ has transverse self-intersections $p\in M$ of multiplicity at most
\[
    m(f,p)\leq\frac{d}{d-c}.
\]
Indeed, if $m=m(f,p)$ transverse tangent spaces of $C$ pass through $p$, then the sum of codimensions $m(d-c)$ must be at most $d$; otherwise, there would exist a normal direction in which one of them could be pushed off the others, decreasing the multiplicity.

Hence, assuming $\frac{d}{d-c}<3$, or equivalently, that we are in the \textit{metastable range}:
\begin{equation}\label{eq:metastable}
    c< \frac{2}{3}d,
\end{equation}
a generic $f$ only has transverse double points. In this setting Haefliger generalized Whitney's double point count $I_f\in\Z$, using the observation that this number can be viewed as the intersection of the image of the restricted square of $f$:
\begin{equation}\label{eq:restr-square}
    f\tm f|\colon (C\tm C)\setminus\Delta_C\ra M\tm M
\end{equation}
with the diagonal $\Delta_M$. This information can be repackaged as follows.

The map $f\tm f\colon C\tm C\to M\tm M$ is $\Z/2$-equivariant for the action that swaps the two coordinates; we write
\[
	f\tm f\in \Map_\partial^{\Z/2}(C\tm C,M\tm M).
\]  
Then for the diagonals we have
\[
    \Delta_C \subseteq  (f\tm f)^{-1} (\Delta_M),
\]
and this is an equality if and only if $f$ is an embedding. 

If $F\in \Map_\partial^{\Z/2}(C\tm C,M\tm M)$ satisfies $\Delta_C =  F^{-1} (\Delta_M)$ we say that $F$ is \textit{isovariant}; it is \textit{strongly isovariant} if additionally for each $x\in\Delta_C$ the differential $dF_x\colon T_x(C\tm C)\to T_{F(x)}(M\tm M)$ is a $\Z/2$-isovariant linear map. Denote by $\Map^{siv}_\partial(C\tm C,M\tm M)$ the subspace of $\Map_\partial^{\Z/2}(C\tm C,M\tm M)$ consisting of strongly isovariant maps, and note that we would like to homotope $f\tm f$ into it.

\begin{theorem}[Haefliger]\label{thm:Haef}
    If $c< \frac{2}{3}d-1$, then we have a 1-cartesian square
    \[\begin{tikzcd}
    \Embp(C,M)\rar[hook]\dar 
    & \Map_\partial(C,M)\dar{f\mapsto f\tm f}\\
    \Map^{siv}_\partial(C\tm C,M\tm M)\rar[hook]
    & \Map_\partial^{\Z/2}(C\tm C,M\tm M)
\end{tikzcd}\]
\end{theorem}

This means that the map from $\Embp(C,M)$ to the homotopy pullback of the rest of the diagram,
\[
    \Embp^{Haef}(C,M)\coloneqq\{(f,F,\gamma):\gamma(0)=F, \gamma(1)=f\tm f\},
\]
is a bijection on $\pi_0$ and a surjection on $\pi_1$. 
Equivalently, for a fixed map $f\colon C\to M$ the map between horizontal homotopy fibers
\[
    \hofib_f(\Embp\to\Map_\partial) \ra 
    \hofib_{f\tm f}(
    \Map^{siv}_\partial\to \Map_\partial^{\Z/2}
	)
\]
is 1-connected. In particular, isotopy classes of embeddings with a given homotopy to $f$ are in bijection with homotopy classes of pairs $(F,\gamma)$ for a strongly isovariant map $F\colon C\tm C\to M\tm M$ and a path $\gamma$ from $F$ to $f\tm f$ in $\Z/2$-equivariant maps. In other words, Haefliger proves that any such choice of $(F,\gamma)$ can be used to homotope $f$ into an embedding, that is, resolve its double points.

\subsection{Dax}\label{app:Dax}

Haefliger's double point elimination methods were further developed by Jean-Pierre Dax~\cite{Dax} in the 1970s. Independently, Hatcher and Quinn~\cite{Hatcher-Quinn} used similar methods for disjunction problems --- making two (families of) submanifolds disjoint. In fact, Goodwillie, Klein, and Weiss~\cite{GKW} sketch a proof of a disjunction result and the derivation of Dax's results from it, in the following form.

\begin{theorem}[Goodwillie--Klein--Weiss]\label{thm:GKW}
    For any $1\leq c\leq d$ the square in Theorem~\ref{thm:Haef} is $(2d-3-3c)$-cartesian. That is, the map 
    \[
    \Embp(C,M)\ra\Embp^{Haef}(C,M)
    \]
     is $(2d-3-3c)$-connected.
\end{theorem}
Where does the expression $2d-3-3c$ come from?
Consider a map
    \[
        F\colon \Big(I^n,\, I^{n-1}\tm\{0\},\, \partial I^{n-1}\tm I\cup I^{n-1}\tm \{1\}\Big) \ra
        \Big( \Immp(C,M),\, \Embp(C,M),\, e \Big)
    \]
and its track
    \[
        \wt{F}\colon I^n\tm C\ra I^n\tm M,\quad \wt{F}(v,x)=(v,F(v)(x)).
    \]
For $\wt{F}$ to be an immersion with self-intersections of multiplicity at most $2$ it suffices to have
\[
    n+c\leq \frac{2}{3}(n+d)-1
\]
as in Theorem~\ref{thm:Haef}, or equivalently
\[
    n\leq 2d-3-3c.
\]
Then, the argument from the previous section about generic immersions $f\colon C\imra M$ translates to the following multi-parameter version.
\begin{lemma}[Dax]\label{lem:Dax}
    If $d-2c\geq0$ and $0\leq n\leq 2d-3-3c$ then
    $F$ can be approximated, relative to the boundary, by a perfect map: a smooth map $F$ whose track
    has no triple points and only transverse double points.
\end{lemma}

Using this lemma Dax generalized Whitney's self-intersection number to invariants that compute the relative homotopy groups $\pi_n(\Immp(C,M),\Embp(C,M),e)$ for $0\leq n\leq 2d-3-3c$, as follows.

First, \eqref{eq:restr-square} tells us that $F$ is a perfect map if and only if the restricted square of its track
\[
    \wt{F}\tm\wt{F}|\colon (I^n\tm C)^2\setminus\Delta_{I^n\tm C}\ra (I^n\tm M)^2
\]
is transverse to the diagonal $\Delta_{I^n\tm M}$. Consider the double point preimage set
\[
    \wt{\Delta}_{\wt{F}}\coloneqq (\wt{F}\tm\wt{F}|)^{-1}(\Delta_{I^n\tm M}).
\]
Similarly to the discussion after \eqref{eq:restr-square}, homotoping $F$ to have image in embeddings corresponds to homotoping $\wt{F}$ in order to make $\wt{\Delta}_{\wt{F}}$ empty. This is a manifold of dimension $2(n+c)-n-d=n-(d-2c)$, and homotopies of $\wt{F}$ change it by a bordism. To get the right answer we need to consider bordism classes of such manifolds not over the space $C\tm C\setminus\Delta_C$, but instead over its ``homotopical resolution'':
\[
        \wt{E}_{e}^\gamma\coloneqq\Big\{ (x_1,x_2,\rho)\in (C\tm C\setminus\Delta_C)\tm\Map([0,1],M) : \rho(0)=e(x_1),\,\rho(1)=e(x_2)\Big\}.
\]
Indeed, observe that if $\wt{F}(t_1,x_1)=\wt{F}(t_2,x_2)=p$ then we have one natural path from $e(x_1)$ to $p$ and one from $p$ to $e(x_2)$. Namely, $\wt{F}(t_1-s(1-t_1),x_1)$, for $s\in[-1,0]$ starts at $e(x_1)$ and ends at $\wt{F}(t_1,x_1)=p$, whereas $\wt{F}(t_2+s(1-t_2),x_2)$, for $s\in[0,1]$ starts at $\wt{F}(t_2,x_2)=p$ and ends at $e(x_2)$. Thus, the concatenation of these paths is a path $\rho$ such that the tuple $(x_1,x_2,\rho)$ defines a point in $\wt{E}_{e}^\gamma$.

One last subtlety is that we have to mod out by the $\Z/2$-involutions, because of the ambiguous sheet choice. Namely, define $\Delta_{\wt{F}}$ as the quotient of $\wt{\Delta}_{\wt{F}}$ under the $\Z/2$-involution, and let $E_{e}^\gamma$ be the quotient of $\wt{E}_{e}^\gamma$ by the involution $(x_1,x_2,\rho)\mapsto (x_2,x_1,\rho^{-1})$. Then, as in the above discussion, we have a map
\[
    b_{\wt{F}}\colon \Delta_{\wt{F}}\ra E_{e}^\gamma,
    \; [(t_1,t_2,x_1,x_2)]\mapsto (x_1,x_2,\rho),
    \; \rho\coloneqq 
    \begin{cases} 
        \wt{F}(t_1-s(1-t_1),x_1) & s\in[-1,0],\\
        \wt{F}(t_2+s(1-t_2),x_2), & s\in[0,1].
    \end{cases}
\]
Together with certain vector bundle data not discussed here (see \cite{Dax} or \cite[Sec.4.1.1]{KT-highd}), this defines a bordism element $( \Delta_{\wt{F}},b_{\wt{F}},\nu_{\wt{F}})\in\Omega_{n-(d-2c)}\big(E_{e}^\gamma;\vartheta_{e}\big)$.

\begin{theorem}[Dax]\label{thm:Dax}
    For $d-2c\geq0$ and $0\leq n\leq 2d-3-3c$ there are isomorphisms
    \[\pi_n\Big(\Immp(C,M),\,\Embp(C,M),\,e\Big)
    \overset{\cong}{\ra}
    \Omega_{n-(d-2c)}\big(E_{e}^\gamma;\vartheta_{e}\big).
    \]
\end{theorem}

For $n=2$, $c=1$, $d=4$, one can show $\Omega_0(E_{e}^\gamma;\vartheta_{e})\cong\Z[\pi_1M]$ and recover Theorem~\ref{thm:Dax-dax}; see~\cite{KT-highd}.
\begin{remark}
    One can also relate the Dax invariant to the Wall invariant in another way: in~\cite[Rem.4.19]{KT-highd} we show that when $\dim M=2\dim C>4$ the isomorphism from Theorem~\ref{thm:Dax} agrees with the Wall invariant from §\ref{subsec:wall}:
\[
    \mu\colon
    \faktor{\pi_0\Imm_\partial(C,M)}{\pi_0\Emb_\partial(C,M)}\ra
    \faktor{\Z[\pi_1M]}{g\sim (-1)^{\dim C}g^{-1}}.
\]
\end{remark}

\printbibliography

\end{document}
